\documentclass[10pt,reqno]{amsart}
\usepackage{epsfig,amssymb,amsmath,version}
\usepackage{amssymb,version,graphicx,fancybox,mathrsfs,multirow}
\usepackage{epstopdf}
\usepackage{url,hyperref}
\usepackage{bm}
\usepackage{subfigure}
\usepackage{color,xcolor}
\usepackage{cases}
\usepackage{mathtools}
\usepackage{extarrows}
\usepackage{bbm}
\usepackage{float}
\usepackage{booktabs}
\usepackage{diagbox}

\allowdisplaybreaks[4]

\catcode`\@=11 \theoremstyle{plain}
\@addtoreset{equation}{section}   
\renewcommand{\theequation}{\arabic{section}.\arabic{equation}}
\@addtoreset{figure}{section}
\renewcommand\thefigure{\thesection.\@arabic\c@figure}
\renewcommand{\thefigure}{\arabic{section}.\arabic{figure}}
\newtheorem{thm}{\bf Theorem}

\newenvironment{theorem}{\begin{thm}} {\end{thm}}
\newtheorem{cor}{\bf Corollary}
\newtheorem{pro}{Proposition}[section]

\newtheorem{lmm}{\bf Lemma}

\newcommand{\sign}{\text{sign}}

\theoremstyle{example}

\theoremstyle{remark}
\newtheorem{rem}{\bf Remark}[section]
\theoremstyle{definition}

\numberwithin{table}{section}

\def \ri {{\rm i}}

\def \d{{\rm d}}

\newcommand{\bs}[1]{\boldsymbol{#1}}

\def \rd {{\rm d}}
\def \V {{\mathbb V}}

\renewcommand \wedge \times

\begin{document}
	\bibliographystyle{plain}
	\graphicspath{{./figs/}}
	
\title[Integral fractional Laplacian] {On diagonal dominance of FEM stiffness matrix of fractional Laplacian and maximum principle preserving schemes for fractional Allen-Cahn equation}
 \author[H. Liu, \, C. Sheng, \,  L. Wang \,  \& \,  H. Yuan]{Hongyan Liu${}^{1}$, \; Changtao Sheng${}^{2}$,  \; Li-Lian Wang${}^{2}$ \; and \; Huifang Yuan${}^{3}$}

		
\thanks{${}^{1}$School of Mathematical Sciences, University of Electronic Science and Technology of China, Chengdu Sichuan 611731, China.  Email: hyliu@std.uestc.edu.cn (H. Liu).\\
\indent ${}^{2}$Division of Mathematical Sciences, School of Physical and Mathematical Sciences, Nanyang Technological University, 637371, Singapore. The research of the  authors is partially supported by Singapore MOE AcRF Tier 2 Grants:  MOE2018-T2-1-059 and MOE2017-T2-2-144. Emails: ctsheng@ntu.edu.sg (C. Sheng) and  lilian@ntu.edu.sg (L. Wang).\\
	\indent ${}^{3}$Department of Mathematics, Southern University of Science and Technology, Shenzhen, 518055, China, and School of Mathematics and Statistics, Wuhan University, Wuhan, 430072, China.
	Email:  yuanhf@sustech.edu.cn (H. Yuan).\\
	\indent The first and last two authors would like to thank NTU  for hosting their visits devoted to this collaborative work.
         }

 \begin{abstract}
In this paper, we  study  diagonal dominance of the  stiffness matrix resulted from the piecewise linear finite element discretisation
of the integral fractional Laplacian under global homogeneous Dirichlet boundary condition in one spatial dimension.  We first derive the exact form of this matrix in the frequency space which is extendable to multi-dimensional rectangular elements.   Then we give the complete answer when the stiffness matrix can be strictly diagonally dominant.  As one application, we apply this notion to the construction of maximum principle preserving schemes for the fractional-in-space Allen-Cahn equation, and provide ample numerical results to verify our findings.
%
\end{abstract}
%
	
\maketitle

\vspace*{-15pt}
\section{Introduction}

The study of diagonal dominance of a matrix has been a research  subject of longstanding  interest in numerical linear algebra and  numerical analysis (cf. \cite{Roger1985,Golub1996}).
On one hand,  this type of structured matrices enjoy  appealing properties,   such as  stable Gaussian elimination without pivoting and guaranteed convergence of Jacobi and Gauss-Seidel iterations among others  (cf. \cite{Alfa2002,george2004gaussian,farid2011notes}).  On the other hand, numerical methods for solving PDEs are a rich source of many linear systems whose coefficient matrices form diagonal dominant matrices (cf. \cite{urekew1993importance,Quarteroni94,Tveito2009}). One well-worn example is the matrix resulted from the piecewise finite element discretization of $u''(x)$  with homogeneous Dirichlet boundary conditions on a uniform partition of a finite interval. However, this property is unknown to date for the  fractional counterpart $(-\Delta)^{s}u(x)$.
 The main purpose of this paper is to provide a complete answer to this and discuss one of its applications.

 We consider a piecewise linear finite element approximation of the  fractional Poisson equation on the finite interval $\Omega=(a,b)$ with $s\in (0,3/2)$:
\begin{equation}
(-\Delta)^{s} u(x)=f(x),\;\;\; x\in\Omega; \quad u(x)=0,\;\;\; x\in\Omega^{c},
\end{equation}
where the integral fractional Laplacian operator takes the form
\begin{equation}\label{fracLap-defn}
(-\Delta)^s  u(x)=C_s \,{\rm p.v.}\! \int_{\mathbb R} \frac{ u(x)- u(y)}{|x-y|^{1+2s}} \rd y,\quad C_s:=\frac{2^{2s}s\Gamma(s+1/2)}{\sqrt{\pi}\Gamma(1-s)},
\end{equation}
or equivalently by the Fourier transform:
\begin{equation}\label{viafouriertransform}
  (-\Delta)^{s} u(x)={\mathscr F}^{-1}\big[|\xi|^{2s} \mathscr{F}{\left[ u\right]}({\xi})\big](x).
\end{equation}
Let $\{\phi_{j}\}_{j=1}^{N-1}$ be a set of $C^0$-piecewise linear nodal basis associated with a uniform partition of $\Omega$ with mesh size $h$. 
Different from the computation in the physical space based on \eqref{fracLap-defn} (cf. \cite{wang2019finite}), we evaluate the entries of the fractional stiffness matrix $\bs S$ in the frequency space using \eqref{viafouriertransform}:
\begin{equation}\label{Selement}
S_{kj}=S_{jk}=\big((-\Delta)^{s/2} \phi_j,  (-\Delta)^{s/2}  \phi_k \big)_{\mathbb R}=\int_{\mathbb R}|\xi|^{2s} {\mathscr F}[ \phi_j](\xi)\overline{{\mathscr F}[\phi_k](\xi)} \,{\rm d}\xi,
\end{equation}
which leads to the explicit expression  of  this symmetric Toeplitz matrix  (see Theorem \ref{th21}).  Remarkably, this approach can be extended to rectangular tensorial finite elements in two or three-dimensional rectangular or L-shaped domains by reducing $2d$-dimensional integrals into one- or two-dimensional integrals using polar or spherical coordinates  (which we shall report in a separate paper).  It is important to remark  that  the computation of the stiffness matrix in two-dimensions on unstructured meshes is much involved (cf.  \cite{acosta2017short,Ainsworth2018}).
It is  also noteworthy of the recent works on  quadrature-based finite difference  methods for integral fractional Laplacian on regular domains \cite{tian2015class,duo2018novel,duo2019accurate}. 

With the explicit form of $\bs S$  at our disposal, we can rigorously show that (see Theorem \ref{DM01}): (i) when the fractional order $s\in (s_0,1]$ with $s_0\approx 0.2347,$ the stiffness matrix $\bs S$ is strictly diagonally dominant with positive diagonal entries;  (ii) for $s\in (1, 3/2),$
$\bs S$ is non-diagonally dominant, and each diagonal entry is strictly smaller than the summation of   other entries (in magnitude) in the same row (except for the first and last rows); and (iii) for $s\in (0,s_0),$ there exists an $N_0(s)$ such that if $N<N_0(s),$   the strict diagonal dominance still holds. In fact, the smallest $N_0(s)$ is around $155$ attained  at $s_*\approx 0.13$ and then increases rapidly as the distance $|s-s_*|$ (for $s\in (0,s_0))$ increases (see  Table \ref{TabN0} and Figure \ref{t1sN0} (b)).


The second purpose of this paper  is to  apply  the notion of diagonal dominance
to  the construction of maximum principle preserving schemes for the fractional-in-space Allen-Cahn equation with spatial  finite element discretisation. More precisely, we consider
\begin{equation}\label{fAC00}
\begin{cases}
u_t + \epsilon^{2}(-\Delta)^s u + f(u) = 0, \quad & x \in \Omega,\;\; t\in(0,T],\\
u(x,t)=0, \quad & x \in \Omega^c=\mathbb R\setminus \Omega,\;\;t\in [0,T],\\
u(x, 0)=u_{0}(x),\quad & x \in \Omega,
\end{cases}
\end{equation} 
where $f(u) = F^\prime (u)$ with
\begin{equation}\label{Fu00}
F(u) = \frac{u^2(u-1)^2}{4}\;\;\; {\rm so}\;\;\; f(u)=\frac{u(u-1)(2u-1)}{2}.
\end{equation}
Different from the usual double-well potential with minima at $u=\pm 1$, i.e., $F(u)=(u^2-1)^2/4$,  the modified $F(u)$
has minima at $u=0,1$ (cf.  \cite{li2017space,liu2018time,duo2019fractional}), in view of the  global ``boundary condition'' imposed on $\Omega^{c}$.  There has been much recent interest in numerical solutions of  fractional-in-space models but  with possibly different definitions of the fractional operator. For example,
 Burrage et al. \cite{burrage2012efficient} considered the solutions of  fractional diffusion equations with the ``discrete'' fractional Laplacian obtained by first finding a matrix representation, $\bs A$, of the Laplacian (by the finite element) and raising it to the same fractional power $\bs A^s$.   Bueno-Orovio et al. \cite{bueno2014fourier}  considered the spectral fractional Laplacian and proposed  Fourier spectral methods.
 In Hou et al. \cite{hou2017numerical}, Crank-Nicolson finite difference method for fractional-in-space Allen-Cahn equation with the fractional derivative:
\begin{equation}\label{oneDD}
\mathcal{L}^{\alpha}_x u(x):=\frac{1}{-2\cos(\frac{\pi\alpha}{2})}\big({}_{a}D_{x}^{\alpha}u+{}_{x}D_{b}^{\alpha}u\big)(x),\quad\alpha\in(1,2),
\end{equation}
where ${}_{a}D_{x}^{\alpha}u$ and ${}_{x}D_{b}^{\alpha}$ denote the left and right Riemann-Liouville fractional derivatives defined on $\Omega$, and the finite-difference matrix with the usual homogeneous  boundary condition: $u|_{\partial\Omega}=0,$ was derived from \cite{tian2015class}. The method in \cite{hou2017numerical} can be directly extended to the  multi-dimensional model with the directional fractional Laplacian $\mathcal{L}^{\alpha}_x u(x)=(\mathcal{L}^{\alpha}_{x_1}+\mathcal{L}^{\alpha}_{x_2} + \mathcal{L}^{\alpha}_{x_3}) u(x)$ on $\Omega^3$ with $u(x)|_{\partial \Omega^3}=0,$ in light of the tensorial nature of the operator and domain. It is known   that $\mathcal{L}^{\alpha}_x u(x)=(-\Delta)^{\alpha/2} u(x)$ on $\Omega$, when $u=0$ on $\Omega^c$ and $\alpha\in (0,2)$ but $\alpha\not=1$.  However, under the local boundary condition: $u|_{\partial \Omega}=0,$ they are different.  Recently, Duo and Wang  \cite{duo2019fractional} proposed quadrature-based finite difference
method  for \eqref{fAC00} with the difference matrix obtained earlier in  \cite{duo2019accurate}, where the approximation error   $(-\Delta)^s u-(-\Delta)_h^su$ is of order $h^2$ in $L^\infty$-sense.  Wang et al. \cite{wang2019finite}  studied the finite element methods for the fractional-in-space Cahn-Hilliard equation.
It is also noteworthy that there is a growing interest in time-fractional Allen-Cahn model (cf. \cite{du2019time,liu2018time,tang2019energy,zhao2019power,liao2019second}).  Needless to say, the development of efficient numerical methods for the integer order Allen-Cahn/Cahn-Hilliard equations and more general phase-field models is continuously
attracting much research attention. One can  refer to the review paper
 \cite{shen2019new} and  the book chapter \cite{du2020phase} for the state-of-the-art and  comprehensive lists of references.

Different from the very limited existing works, we consider  finite element discretisation in space with a modification similar to that in Xu et al.  \cite{xu2019stability}, and propose the  semi-implicit  and Crank-Nicolson schemes  as advocated in  \cite{tang2016implicit,tang2020revisit} for the integer-order Allen-Cahn equation.  We show that the proposed schemes preserve maximum principle and energy dissipation (for $s\in (s_0, 1)$) at the discrete level.  Though we  focus on one dimensional in space, the methods can be extended to multiple dimensions with the
 directional fractional Laplacian  $\mathcal{L}^{\alpha}_x u(x)$ and global homogeneous Dirichlet boundary condition. However,
 the construction of this type of schemes for the integral fractional Laplacian in multiple dimensions is still open, though such properties can be shown at the continuous level.

The rest of this paper is organised as follows. 
In Section \ref{sect2:FEM}, we present  the exact form of the FEM stiffness matrix based on the Fourier definition with implementation in the Fourier space.  More importantly, we prove the main result on the diagonal dominance of this matrix.
In Section \ref{sect3:AC}, we propose the semi-implicit and modified FEM schemes for the fractional-in-space Allen-Cahn equation, and show that they preserve the maximum principle and energy dissipation.
 In Section \ref{sect4:numer}, we provide ample numerical results to  support the theoretical results.
 The final section is for some concluding remarks.

\section{Finite element method for fractional Laplacian}\label{sect2:FEM}
\setcounter{equation}{0} \setcounter{lmm}{0} \setcounter{thm}{0}
In this section, we derive the explicit stiffness matrix of the $C^0$-piecewise linear FEM for  the fractional Laplacian using the frequency domain. 
More importantly, we will study the diagonally dominant properties of the stiffness matrix for piecewise linear FEM.

\subsection{Finite element method}


Consider a uniform partition of the interval $\Omega=(a,b)$:
$$x_j=a+jh,\quad 0\le j\le N,\quad h= 2/N.$$
The piecewise linear FEM basis  is given by
\begin{equation}\label{p1fembasis}
\phi_j(x)=\begin{cases}
\frac{x-x_{j-1}}{h},\quad & {\rm if}\;\; x\in (x_{j-1},x_j),\\[2pt]
\frac{x_{j+1}-x}{h}, \quad & {\rm if}\;\; x\in (x_{j},x_{j+1}),\\[2pt]
0,\quad & \text{elsewhere on}\;\; \mathbb R.
\end{cases}
\end{equation}
Correspondingly, we define the piecewise linear finite element space
\begin{equation}\label{Vhh}
\V_{\!h}={\rm span}\{\phi_j(x),\; 1\leq j\leq N-1\}.
\end{equation}
and  intend to evaluate the $(N-1)\times (N-1)$ fractional stiffness matrix $\bs S$ with the entries
\begin{subequations} \label{uvsh}
\begin{align}
S_{kj}&=S_{jk}=
 \frac{C_{s}} 2 \int_{\Omega} \int_{\Omega}
 \frac{(\phi_j(x)-\phi_j(y))(\phi_k(x)-\phi_k(y))}{|x-y|^{d+2s}}\,{\rm d} x {\rm d}  y\nonumber\\[4pt]
 & \quad  +C_{d,s}  \int_{\Omega} \bigg(\int_{\Omega^c}
\frac 1 {|x-y|^{d+2s}}  {\rm d} y\bigg)  \phi_j(x) \phi_k(x)\, {\rm d} x \label{phyform}\\[4pt]
&=\int_{\mathbb R}|\xi|^{2s} {\mathscr F}[ \phi_j](\xi)\overline{{\mathscr F}[\phi_k](\xi)} \,{\rm d}\xi, \label{freqform}
\end{align}
\end{subequations}
for $1\le k,j \le N-1.$ The representation  \eqref{phyform} corresponds to the implementation in the  physical space, while the formula \eqref{freqform} is implemented in the frequency space.   

The following formula on the Fourier transform of the FEM basis  plays an important role in the evaluation of $\bs S$.
\begin{lmm}\label{Dunfordfemp1} Let $\{\phi_j\}$ be the FEM basis given in \eqref{p1fembasis}. Then we have
\begin{equation}\label{phinfem}
 {\mathscr F}[ \phi_j](\xi)=\frac {2h} {\sqrt{2\pi}} \frac{1-\cos(h\xi)}{(h\xi)^2}  {e^{-\ri x_j\xi} },\quad \forall\, \xi\in {\mathbb R},\;\; 1\le j\le N-1.
\end{equation}
\end{lmm}
\begin{proof}
Using \eqref{p1fembasis} and integration by parts, we obtain from direct calculation that
\begin{eqnarray}
 &&{\mathscr F}[ \phi_{j}](\xi)=\frac{1}{\sqrt{2\pi}}\int_{\mathbb R}  \phi_j (x)e^{-\ri\xi x} {\rm d}x = \frac{1}{\sqrt{2\pi}} \int^{x_{j+1}}_{x_{j-1}} \phi_j (x)e^{-\ri\xi x} {\rm d}x\nonumber
  \\&&\quad=\frac{1}{h\sqrt{2\pi}} \Big\{\int^{x_j}_{x_{j-1}} (x-x_{j-1}) e^{-\ri\xi x} {\rm d}x+ \int^{x_{j+1}}_{x_{j}} (x_{j+1}-x) e^{-\ri\xi x} {\rm d}x\Big\}\nonumber
 \\&&\quad=\frac{1}{h\sqrt{2\pi}} \Big\{\frac{-h}{\ri \xi}e^{-\ri x_j\xi}+\frac{1}{\xi^2}(e^{-\ri x_j\xi}-e^{-\ri x_{j-1}\xi})\Big\}\nonumber
 \\&&\qquad+ \frac{1}{h\sqrt{2\pi}} \Big\{\frac{h}{\ri \xi}e^{-\ri x_j\xi}-\frac{1}{\xi^2}(e^{-\ri x_{j+1}\xi}-e^{-\ri x_{j}\xi})\Big\}\nonumber
 \\&&\quad=\frac{-1}{h\sqrt{2\pi}} \Big\{\frac{e^{-\ri x_{j-1}\xi}-2e^{-\ri x_j\xi}+e^{-\ri x_{j+1}\xi}}{\xi^2}\Big\}\nonumber
 \\&&\quad=\frac{-e^{-\ri x_j\xi}}{h\sqrt{2\pi}} \Big\{\frac{e^{-\ri h\xi}+e^{\ri h\xi}-2}{\xi^2}\Big\}=\frac{2e^{-\ri x_j\xi}}{h\sqrt{2\pi}} \Big\{\frac{1-\cos(h\xi)}{\xi^2}\Big\}.\nonumber
\end{eqnarray}
This ends the proof.
\end{proof}
With the aid of Lemma \ref{Dunfordfemp1}, we can obtain the entires of stiffness matrix $\bs S$ explicitly. Here, we sketch the derivation in Appendix \ref{AppendixA} to avoid distraction from the main result.
\begin{theorem}\label{th21}
For $s\in(0,\frac32)$,  
the FEM stiffness matrix $\bs S=(\bs S_{kj})$ is a symmetric Toeplitz matrix given by 
\begin{equation}\label{StiffMatrix}
\bs S=
  \frac{h^{1-2s} } {2  \Gamma(4-2s) \cos(s\pi)} \begin{bmatrix}
   t_{0} &\hspace{-4pt} t_{1} & t_2 & \cdots & t_{N-4}& t_{N-3} & t_{N-2} \\[1pt]
   t_{1} &\hspace{-4pt} t_{0} & t_{1} &\hspace{-4pt} \ddots & \cdots & t_{N-4}  & t_{N-3}\\[-1pt]
   t_{2} &\hspace{-4pt} t_{1} & t_{0} & \ddots & \hspace{-4pt}\ddots &\vdots   & t_{N-4} \\[0pt]
   \vdots& \hspace{-4pt}\ddots& \hspace{-4pt}\ddots & \ddots   & \hspace{-4pt}\ddots &\hspace{-4pt} \ddots & \vdots \\[2pt]
   t_{N-4}   &  \vdots   &\hspace{-4pt} \ddots   &\hspace{-4pt} \ddots  &\hspace{8pt} t_0 &t_1  & t_2   \\[-2pt]
   t_{N-3}  & t_{N-4} & \cdots   &\hspace{-4pt}\ddots &\hspace{8pt}  t_{1}  & t_{0} & t_{1} \\[3pt]
   t_{N-2}  & t_{N-3} & t_{N-4}  & \cdots &\hspace{8pt} t_2  & t_{1} & t_{0}
  \end{bmatrix},
\end{equation}
which
 is generated by the vector $(t_0, t_1, \cdots, t_{N-2})$ in the first  row or column  of $\bs S$ with
\begin{equation}\begin{split}\label{gpa00}
 t_p=  \sum_{i=-2}^2 c_i |p+i|^{3-2s},\quad c_0=6,\;\; c_{\pm 1}=-4,\;\; c_{\pm 2}=1.
 \end{split}\end{equation}
 In particular, if $s=1/2$,   the entries of $\bs S$ should be  obtained by
\begin{equation}\label{StiffMatrix01/2}
 S_{kj}= \frac 1 4 \lim_{s\to \frac 1 2} \frac{t_{p}}{\cos(s\pi)}=\frac 1 {2\pi} \sum_{i=-2}^2 c_i (p+i)^2 \ln |p+i|, \;\;\;\;  p=|k-j|,
\end{equation}
where we should  understand that  $(p+i)^2\ln |p+i|=0$ when $p+i=0.$
%
\end{theorem}

\begin{rem}{\em
Letting $s\to 0$ and $s\to 1$, the matrix $\bs S$ in Theorem \ref{th21} reduces to  the usual FEM mass matrix $\bs M$ and  stiffness matrix $\bs S$:
\begin{equation*}\label{skj1200}
M_{k j}=h\begin{cases}{2/3,} & {j=k}, \\ {1/6,} & {j=k \pm 1},  \\ {0,} & {\text {otherwise}},
\end{cases}
\qquad
S_{kj}
 =\frac{1}{h}
 \begin{cases}{2,}\;\; &   {j=k}, \\
 {-1,} & {j=k \pm 1}, \\ {0,} & {\text {otherwise,}}
 \end{cases}
\end{equation*}
respectively.  \qed
}
\end{rem}

\begin{rem}{\em We point out that in 1D, it is feasible to compute $S_{jk}$ in the physical space using  \eqref{phyform}  {\rm (}cf.
\cite{tian2015class,wang2019finite}{\rm).}  However,  the implementation in the physical space becomes very complicated    {\rm(}cf. \cite{acosta2017short,Ainsworth2018}{\rm)}. In fact,  the frequency domain approach can be extended to multiple dimensional uniform rectangular elements, which leads to computing a one-dimensional integral on $(0,\pi/2)$ rather than $2\times 2$-dimensional  integrals  in two dimensions.  We shall report this in a  separate work. \qed
}
\end{rem}

\subsection{Diagonal dominance of the stiffness matrix}
We first make necessary preparations through the following two lemmas.
\begin{lmm}\label{elem1}   Let $s\in (0,\frac 3 2)$ and $s\not=\frac 1 2.$
\begin{itemize}
\item[(i)] The element
\begin{equation}\label{t1}
t_1=t_1(s)=7+3^{3-2s}-2^{5-2s}
\end{equation}
has a unique root $s_0\approx 0.2347$ in the interval $(0,\frac 1 2).$  Moreover, we have
\begin{equation}\label{t1sign}
\begin{cases}
 t_1>0, \;\; &{\rm if}  \;\; s\in (0, s_0)\cup(\frac 1 2, \frac 32),\\
 t_1<0, \;\; &{\rm if}  \;\; s\in (s_0,\frac 1 2).
   \end{cases}
\end{equation}
\item[(ii)] For $p\geq 2$, we have
\begin{equation}\label{ppropA}
\begin{cases}
t_p<t_{p+1}<0, \;\;\;\; & {\rm if}\;\;s\in(0,\frac 12)\cup(1,\frac 3 2), \\
t_p>t_{p+1}>0, \;\;\;\;  & {\rm if}\;\;s\in(\frac 1 2,1),\\
t_p=0,\;\;  & {\rm if}\;\;s=1.
\end{cases}
\end{equation}
\end{itemize}
%
\end{lmm}
\begin{proof}
(i) By direct calculation, we find  $t_1'(s)=8 (\ln 2)\, 2^{3-2s} -2(\ln 3)\, 3^{3-2s},$  which has a unique root
\begin{equation}\label{sstar}
s^\ast=\frac 3 2 -\frac{\ln (8 \ln 2)-\ln(2\ln 3)}{2(\ln 3-\ln 2)}\approx 0.3584.
\end{equation}
Moreover, $t_{1}(s)$ is descending in $(0,s^*),$ but asending in $(s^*,\frac32)$.
As $t_1(0)=2$ and $t_1(s^*)\approx -0.1856,$ $t_1(s)$ has a unique root in $(0,s^*).$ Using a root-finding method (e.g., the bisection method), we can easily find $s_0\approx 0.2347.$ Note that $s=\frac 1 2$ is the other unique root of $t_1(s)$ in the interval $(s^*,\frac32).$ Then we have the property
\eqref{t1sign} (cf. Figure \ref{t1sN0}(a)).

\smallskip
(ii) We next consider $p\ge 2$.    
Denote $\alpha = 3-2s$, and rewrite $t_p$ in \eqref{gpa00} as
\begin{equation}\begin{split}\label{gpa}
t_{p}& = p^{\alpha}\Big\{ 6 - 4\Big\{ \Big{(}1+\frac{1}{p}\Big{)}^{\alpha} +\Big{(}1-\frac{1}{p}\Big{)}^{\alpha}\Big\} + \Big\{\Big{(}1+\frac{2}{p}\Big{)}^{\alpha} + \Big{(}1-\frac{2}{p}\Big{)}^{\alpha} \Big\} \Big\}
=\sum_{n = 2}^{\infty}\frac {c_{n}^{(\alpha)}}{ p^{2n-\alpha}},
\end{split}\end{equation}
where  we used the Taylor expansion of $(1+x)^\alpha$, and
$$
  c_{n}^{(\alpha)} :=
 (2^{2n+1}-2^3) \frac{\alpha(\alpha-1)\cdots(\alpha-2n+1)}{(2n)!}.
$$
As $\alpha=3-2s\in (0,3)$  and $\alpha \neq 2$, we have
$$
\sign(c_{n}^{(\alpha)})=-\sign((\alpha-1)(\alpha-2))= -\sign((2s-1)(s-1)).
$$
Note that $t_p$ has the same sign as  $c_{n}^{(\alpha)}.$
Thus, if    $s\in(0,1/2)\cup(1,3/2)$, then $c_{n}^{(\alpha)}<0,$  so $t_p<0$ and
$$
 \frac {c_{n}^{(\alpha)}}{ p^{2n-\alpha}} < \frac {c_{n}^{(\alpha)}}{ {(p+1)}^{2n-\alpha}}, \;\;\; {\rm so} \;\;\;
 t_{p}<t_{p+1}.
 $$
 On the other hand, if $s\in(\frac12,1)$, then  $c_{n}^{(\alpha)}>0,$ so $t_p<0,$ and
$$
 \frac {c_{n}^{(\alpha)}}{ p^{2n-\alpha}} > \frac {c_{n}^{(\alpha)}}{ {(p+1)}^{2n-\alpha}}, \;\;\; {\rm so} \;\;\;
 t_{p}>t_{p+1}.
 $$
 By \eqref{direcal}, we have $t_p=0$ for $s=1$ and $p\ge 2.$
%
Thus, the property \eqref{ppropA} holds.
\end{proof}

\begin{figure}[!th]
\centering
\subfigure[Graph of $t_{1}(s)$]{
\includegraphics[width=0.45\textwidth]{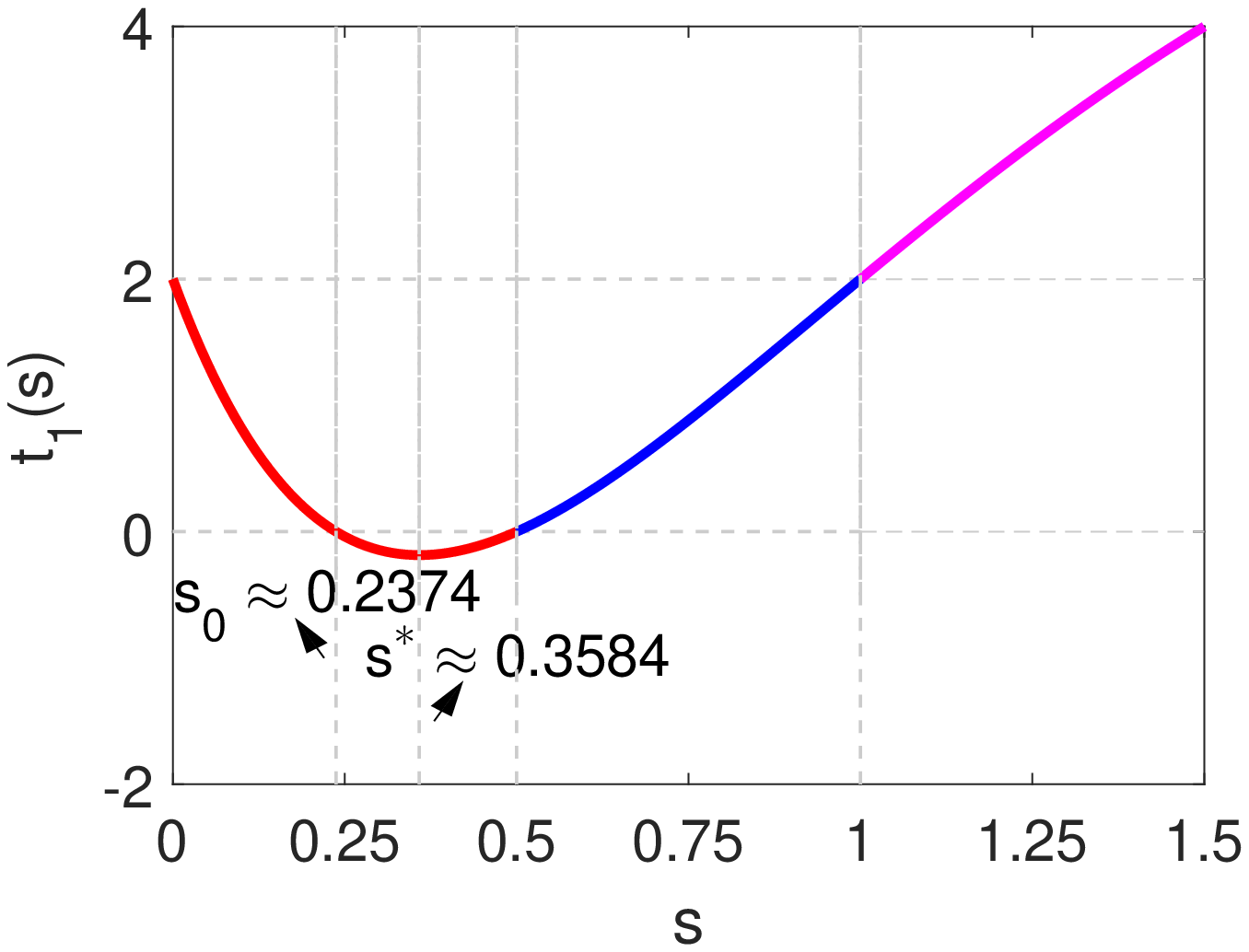}}\quad
\subfigure[Graph of $N_0(s)$]{
\includegraphics[width=0.45\textwidth]{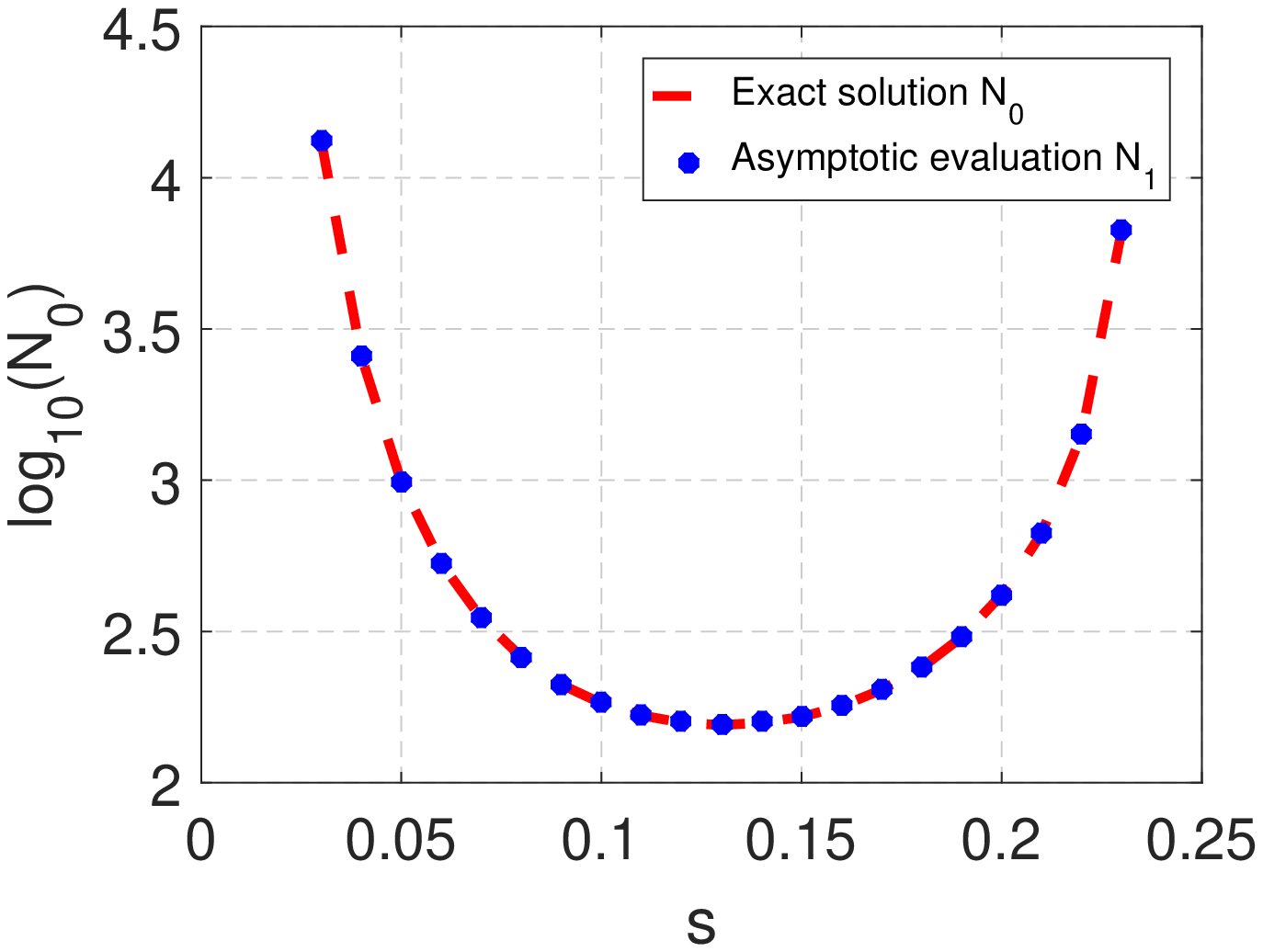}}
\caption{Left: graph of $t_1(s)$. Right: profile of  $N_0(s)$  for $s\in (0,s_0)$, for which the stiffness matrix $\bs S$ is diagonally dominant when $N<N_0(s).$} \label{t1sN0}
\end{figure}

\begin{lmm}\label{elem2} For $s=\frac12$, we denote
\begin{equation}\label{StiffMatrix01/20}
 r_p:=S_{kj}=\frac 1 {2\pi} \sum_{i=-2}^2 c_i (p+i)^2 \ln |p+i|,\quad 0\le p=|k-j|\le N-2.
\end{equation}
Then we have
\begin{equation}\label{tpP12}
r_0>0; \quad r_{p}<r_{p+1}<0,\quad 1\le p\le N-3.
\end{equation}
\end{lmm}
\begin{proof}
Direct calculation from \eqref{StiffMatrix01/20} leads to
\begin{equation*}
\begin{split}
& r_0=\frac 4 \pi \ln 2, \quad r_1=\frac 1 {2\pi}\big(9\ln 3- 16\ln 2\big), \quad
r_2=\frac 2 {\pi}\big(14\ln 2- 9\ln 3\big),\\
& r_3 = \frac{1}{\pi}\Big(27\ln3-72\ln2+\frac{25}{2}\ln5\Big),
\end{split}
\end{equation*}
and one verifies readily that $r_1<r_2<r_3<0.$

For $p\ge 3$, we can rewrite $r_p$ as
\begin{equation}
\begin{split}
r_{p} &= \frac{p^2}{2\pi}\bigg\{ \Big( 1-\frac2p \Big)^2\ln\Big( 1-\frac2p \Big) + \Big( 1+\frac2p \Big)^2\ln\Big( 1+\frac2p \Big)
\\&\quad- 4 \Big( 1-\frac1p \Big)^2\ln\Big( 1-\frac1p \Big) - 4 \Big( 1+\frac1p \Big)^2\ln\Big( 1+\frac1p \Big) \bigg\} \\
& = \sum_{n=2}^{\infty}\frac{c_n}{p^{2n-2}},\quad  \text{with}\ \ \ c_{n} = \frac{2-2^{2n-1}}{\pi n(2n-1)(n-1)},
\end{split}
\end{equation}
using the Taylor expansion
\begin{equation}\label{taylorLn}
(1+x)^2\ln(1+x)=x+\frac 3 2 x^2+2\sum_{n=1}^\infty \frac{(-1)^{n+1} x^{n+2}} {n(n+1)(n+2)}.
\end{equation}
Since $ c_n < 0$ for $n\geq 2$,  we have $r_p < 0$, and
 $$
 \frac {c_{n}}{ p^{2n-2}} <  \frac {c_{n}}{ {(p+1)}^{2n-2}}, \;\;\; {\rm so} \;\;\;
 r_{p}<r_{p+1}.
 $$
This completes the proof.
\end{proof}

With the above preparations, we are now ready to present the main result.
\begin{thm}\label{DM01}
Let $s_{0}\approx 0.2347$ be the root of $t_{1}(s) = 7+3^{3-2s}-2^{5-2s}$ as in Lemma \ref{elem1}, and denote
\begin{equation}\label{AStiffMatrix0}
 A_s:=\frac{1 } {2 \Gamma(4-2s) \cos(s\pi)}.
\end{equation}
Then the stiffness matrix $\bs S=(S_{kj})$ stated in Theorem {\rm\ref{th21}} has the following properties.
\begin{itemize}
\item[(i)] If $s\in [s_0,1)$, we have
\begin{equation}\label{Skkj}
S_{kk}>0; \quad S_{kj}<0,\quad  k\not=j,
\end{equation}
except for    $S_{k,k\pm 1}=0 $ for $s=s_0,$
 and the matrix $\bs S$ is strictly positive diagonally dominated, i.e.,
    \begin{equation}\label{SPDD}
 S_{kk} > \sum_{ k\not=j=1}^{N-1}|S_{kj}|, \quad  1\le  k\le N-1.
    \end{equation}
\item[(ii)]  If $s\in (0,s_0)$, we have $S_{kk}, S_{k, k\pm 1}>0, S_{kj}<0$ for $|k-j|\ge 2,$ and
    \begin{equation}\label{SPDDs0}
    S_{kk} -\sum_{ k\not=j=1}^{N-1}|S_{kj}|>-\frac{4A_s t_1(s)}{h^{2s-1}},  \quad  1\le k\le N-1,
    \end{equation}
    where $t_1\in (0,2)$. However,  the property \eqref{SPDD}
    holds only for $N\le N_0$ with
    \begin{equation}\label{N0form}
    N_0=\Big[2\Big(\frac{\gamma(s)}{t_1(s)}\Big)^{\frac 1 {2s}}\Big],\quad \gamma(s):= (1-2s)(1-s)(3-2s).
    \end{equation}

\item[(iii)]  If $s\in (1,3/2)$, we have $S_{kk}>0, S_{k, k\pm 1}<0, S_{kj}>0$ for $|k-j|\ge 2.$
 The property \eqref{SPDD}  holds only for $k=1$ and $k=N-1,$ but in the contrary, we have that  for $N\ge 4,$
    \begin{equation}\label{SPDD00}
     S_{kk} < \sum_{ k\not=j=1}^{N-1}|S_{kj}|, \quad  2\le k\le N-2.
    \end{equation}

\end{itemize}
\end{thm}
\begin{proof}  We only  prove the results with even $N$, since it is straightforward to prove the statements with odd $N$.
Before we consider different cases of $s,$ we first derive some common properties.

From the matrix form of $\bs S$ in \eqref{StiffMatrix},
we find readily that for $1\le k\le N-1,$
\begin{equation}\label{deltaT}
\begin{split}
& d_k:= |S_{kk}|-\sum_{k\not=j=1}^{N-1}|S_{kj}|=\frac{|A_s|}{h^{2s-1}} \Big( |t_0| -\sum_{p=1}^{k-1}|t_p| - \sum_{p=1}^{N-1-k}|t_p|\Big),
\end{split}
\end{equation}
 and for $1\le k\le N-2,$
\begin{equation}\label{deltaT2}
\begin{split}
& d_{k+1}=d_k+ \frac{|A_s|}{h^{2s-1}} \big(|t_{N-1-k}|-|t_k|\big).
\end{split}
\end{equation}
It is also evident that 
\begin{equation}\label{symdk}
d_k=d_{N-k},\quad 1\le k\le N-1,
\end{equation}
so  it suffices to study $d_k$ with $1\le k\le N/2.$ Moreover, using  the property \eqref{ppropA},  we can derive from   \eqref{deltaT2} that for $s\in (0,\frac 32)$ and $s\not=\frac 1 2,1,$
\begin{equation}\label{monoticityA}
d_2>d_3>\cdots>d_{N/2}.
\end{equation}
In the proof, we shall  check the signs of $d_1$ and $d_{N/2}$ in most of the cases.  For this purpose, we define
\begin{equation}\label{GM}
f_q:=q^{3-2s}, \quad  {\mathcal G}_q:= f_{q-1}-3f_{q}+3f_{q+1} -f_{q+2}.
\end{equation}
Using  \eqref{ppropA} again, we find from direct calculation that for $s\in (\frac 1 2,1),$
\begin{equation}\label{sjk0}
\begin{split}
{\mathcal S}_m&:= \sum_{p = 2}^{m}|t_{p}|  =\sum_{p = 2}^{m} t_{p} = \sum_{p = 2}^{m} (f_{p-2}-4f_{p-1}+6f_p -4f_{p+1}+f_{p+2})
\\&=\sum_{q=0}^3 f_q- 4 \sum_{q=1}^3 f_q +6 \sum_{q=2}^3 f_q -4f_3 -4f_{m-1} \\
&\quad +
6 \sum_{q=m-1}^{m} f_q -4 \sum_{q=m-1}^{m+1} f_q + \sum_{q=m-1}^{m+2} f_q \\[4pt]
& = -3f_1+3f_2-f_3 -f_{m-1}+3f_{m}-3f_{m+1} +f_{m+2} = {\mathcal G}_1-{\mathcal G}_m,
\end{split}
\end{equation}
and for $s\in (0, \frac 1 2)\cup (1,\frac 3 2),$
\begin{equation}\label{sjk01}
\begin{split}
{\mathcal S}_m&:=\sum_{p = 2}^{m}|t_{p}|=- \sum_{p = 2}^{m} t_{p}={\mathcal G}_m-{\mathcal G}_1.
\end{split}
\end{equation}
From \eqref{deltaT} and the above, we have
\begin{equation}\label{d1exp}
\begin{split}
& d_1=\frac{|A_s|}{h^{2s-1}} \bigg( |t_0| -|t_1|-\sum_{p=2}^{N-2} |t_p|\bigg) =\frac{|A_s|}{h^{2s-1}} \left( |t_0| -|t_1|-{\mathcal S}_{N-2}\right), \\
&d_{N/2}  = \frac{|A_{s}|}{h^{2s-1}}\bigg(|t_{0}| -2|t_1|- 2\sum_{p = 2}^{N/2-1}|t_{p}|\bigg)
=\frac{|A_s|}{h^{2s-1}} \left( |t_0| -2|t_1|-2 {\mathcal S}_{N/2-1}\right).
\end{split}
\end{equation}
Note that by \eqref{gpa00}, we have
\begin{equation}\label{t0t1}
t_{0} =2^{4-2s}-8= -8f_{1}+2f_{2},\quad  t_{1} =7+3^{3-2s}-2^{5-2s}= 7f_{1}-4f_{2}+f_{3}.
\end{equation}

It is seen from \eqref{d1exp} that  the sign of $\mathcal G_q$ is important to determine the signs of $d_1, d_{N/2}$, so we rewrite it by using the Taylor expansion:
\begin{equation}\begin{split}\label{gpa22}
\mathcal G_q=  {(q+1)}^{\alpha}\Big\{\Big{(}1-\frac{2}{q+1}\Big{)}^{\alpha}-3
\Big{(}1-\frac{1}{q+1}\Big{)}^{\alpha}-
\Big{(}1+\frac{1}{q+1}\Big{)}^{\alpha}+3\Big\}
=\sum_{n = 3}^{\infty}\frac {\hat c_{n}^{(\alpha)}}{ (q+1)^{n-\alpha}},
\end{split}\end{equation}
where $\alpha=3-2s$ as before, and
\begin{equation}\label{cnalpha}
  \hat c_{n}^{(\alpha)} := (-1)^n (2^{n}-(-1)^n-3) \frac{\alpha(\alpha-1)\cdots(\alpha-n+1)}{n!}.
  \end{equation}
Since $\alpha\in (0,3),$  we have
\begin{equation}\label{signhatcn}
\sign(\mathcal G_q)=\sign(  \hat c_{n}^{(\alpha)})=-\sign ((\alpha-1)(\alpha-2))= -\sign ((2s-1)(s-1)).
\end{equation}

With these, we now proceed with the proof by  considering several cases with different ranges of $s$.
\medskip

\underline{(i)$_1$. $s\in (\frac12, 1)$:}\,   In order to prove Statement-(i), we first consider $s\in (\frac 1 2, 1).$
  Note that $A_s < 0$  (cf. \eqref{AStiffMatrix0}), $t_0=2^{4-2s}-8<0$ and $t_p>0$ with $p\geq 1$ (cf. Lemma \ref{elem1}), so we have
\begin{equation}\label{skj121}
S_{kk} = \frac{A_{s}t_{0}}{h^{2s-1}} > 0; \quad S_{kj} = \frac{A_{s}t_{p}}{h^{2s-1}} < 0,\;\;\;{\rm for}\;\; p = |k-j| \geq 1.
\end{equation}
In view of \eqref{monoticityA},   we only need to show that  $d_{1}>0$ and $d_{N/2}>0.$ By   \eqref{GM}-\eqref{sjk0} and \eqref{d1exp}-\eqref{t0t1},
\begin{equation}\label{deltaT2B00}
\begin{split}
& d_{1} = \frac{A_{s}}{h^{2s-1}}\big( t_0 + t_1+\mathcal S_{N-2} \big)= \frac{A_{s}}{h^{2s-1}}\big(\frac12 t_0-{\mathcal G}_{N-2}\big),
\end{split}
\end{equation}
and
\begin{equation}\label{sjk2}
\begin{split}
d_{N/2} & = \frac{A_{s}}{h^{2s-1}}\big(t_{0} +2t_1+2\mathcal S_{N/2-1}\big)= - \frac{2A_{s}}{h^{2s-1}}\mathcal G_{N/2-1}.
\end{split}
\end{equation}
For $s\in (1/2, 1),$ we know from \eqref{signhatcn} that  ${\mathcal G}_{N-2}, \mathcal G_{N/2-1}>0.$
 As $A_s<0,$ we infer from \eqref{deltaT2B00}-\eqref{sjk2}  that $d_1, d_{N/2}>0,$ so the desired property \eqref{SPDD} holds for $s\in (1/2, 1).$

\medskip
\underline{(i)$_2$.  $s\in [s_0, \frac1 2)$}:\, In this case, we have  $A_s > 0$ (cf. \eqref{AStiffMatrix0}), $t_0=2^{4-2s}-8>0$,  and $t_p<0$ for $p\geq 1$ (except for $t_1=0$ with $s=s_0,$ see  Lemma \ref{elem1}), so their signs are opposite to those of the previous case. As a result,  the property \eqref{Skkj}
 still holds, but with the exceptional case:  $S_{k,k\pm 1}=0,$ if $s=s_0.$
Moreover, we also have the same formulas as  \eqref{deltaT2B00}-\eqref{sjk2} for $d_1, d_{N/2}$, i.e.,
\begin{equation*}\label{D1D2}
d_1=\frac{A_{s}}{h^{2s-1}}\Big(\frac12 t_0-{\mathcal G}_{N-2}\Big),\quad d_{N/2}=  - \frac{2A_{s}}{h^{2s-1}}\mathcal G_{N/2-1},
\end{equation*}
but by \eqref{signhatcn},  we have  ${\mathcal G}_{N-2}, \mathcal G_{N/2-1}<0,$ so
$d_1, d_{N/2}>0.$  Consequently, Statement-(i)  holds for $s\in [s_0,1/2).$

\medskip
\underline{(i)$_3.$    $s=\frac 1 2$:}\,  It is evident that by  Lemma \ref{elem2},  we  have  $S_{kk}>0$
and $S_{kj}<0$ for $k\not=j.$  Moreover,  $d_k$ in \eqref{deltaT} becomes
\begin{equation*}\label{deltaT00}
\begin{split}
& d_k:= |S_{kk}|-\sum_{k\not=j=1}^{N-1}|S_{kj}|= r_0 +\sum_{p=1}^{k-1}r_p + \sum_{p=1}^{N-1-k} r_p.
\end{split}
\end{equation*}
Similarly, we have  $d_k=d_{N-k},$ and
\begin{equation*}\label{readyA12}
d_{k+1}=d_{k}+r_k-r_{N-1-k}< d_{k}, \quad  1\le k\le  N/2-1.
\end{equation*}
In this case, it is only necessary to show that $d_{N/2}>0.$
 Comparing \eqref{StiffMatrix01/20} with \eqref{gpa00}, the formulas  in  \eqref{GM}-\eqref{sjk0} are valid  in place of  $f_q=\frac 1 {2\pi} q^2\ln q$ and $t_p=r_p$ in $\mathcal G_q.$ Like \eqref{sjk2}, we have
 \begin{equation*}\label{dN2T12}
\begin{split}
 d_{N/2}&= - {\mathcal G}_{N/2-1} = \frac1{\pi}\bigg\{ \Big(\frac N2+1\Big)^{2}\ln\Big(\frac N2+1\Big) - 3\Big(\frac N2\Big)^{2}\ln \frac N2
\\
&\qquad  + 3\Big(\frac N2-1 \Big)^{2}\ln\Big(\frac N2-1 \Big) - \Big(\frac N2-2 \Big)^2 \ln\Big(\frac N2-2 \Big) \bigg{\}}
\\
& = \frac{N^2}{4\pi}\bigg\{\Big( 1+\frac2N \Big)^2\ln\Big( 1+\frac2N \Big)
+ 3\Big( 1-\frac2N \Big)^2\ln\Big( 1-\frac2N \Big) - \Big( 1-\frac4N \Big)^2\ln\Big( 1-\frac4N \Big)\bigg\},
\end{split}
\end{equation*}
where we subtracted a summation of five terms with $\ln \frac N 2$ in place of all five  $\ln$s (which is zero).  Using the Taylor expansion \eqref{taylorLn},
we can expand $d_{N/2}$ as
\begin{equation}\label{dN2final}
d_{N/2}=\frac 2 \pi \sum_{n=1}^\infty \frac{2^{n+2}+(-1)^{n+1}-3} {n(n+1)(n+2)}\Big(\frac 2 N\Big)^n,
\end{equation}
which is apparently positive. Thus, Statement-(i) is valid  for $s=\frac 1 2.$

\medskip
\underline{(ii).  $s\in (0,s_0)$:}\,  We now turn to the justification for Statement-(ii).
In this case, we have $A_s>0, t_0>0, t_1>0$ and $t_p<0$ for $p\ge 2.$
We first  show that $d_1>0.$ Indeed, by    \eqref{GM} and \eqref{sjk01}-\eqref{t0t1},
\begin{equation}
\begin{split}
d_1 & = \frac{A_{s}}{h^{2s-1}}\bigg(t_0 - t_1 - \sum_{p=2}^{N-2}|t_p|\bigg)
 = \frac{A_{s}}{h^{2s-1}}\big(9\cdot 2^{\alpha}-2\cdot {3}^{\alpha}-18-\mathcal G_{N-2}\big),
\end{split}
\end{equation}
where  by \eqref{signhatcn},  $\mathcal G_{N-2}<0.$ We now show that  $g(\alpha):=9\cdot 2^{\alpha}-2\cdot {3}^{\alpha}-18>0$ for $\alpha=3-2s\in (\alpha_0,3)$ and $\alpha_{0} = 3-2s_{0}\approx 2.5252$.  One verifies readily that $g(\alpha)$ has one extreme point
\begin{equation*}
\alpha_* = \frac{\ln(9\ln2)-\ln(2\ln3)}{\ln3-\ln2} \approx 2.5736,
\end{equation*}
i.e., $g'(\alpha_*) =0,$ where $g(\alpha)$ attains its local maximum  with $g(\alpha_*)\approx 1.7738$. We can further check that  $g(\alpha)>g(3)=0$ for all $\alpha\in (\alpha_0,3).$
Thus, we have $d_1>0.$
We now consider $d_{N/2}.$
By \eqref{sjk01},
\begin{equation}\label{gN}
\begin{split}
d_{N/2} & = \frac{A_{s}}{h^{2s-1}}\bigg(t_{0} - 2t_{1} - 2\sum_{p = 2}^{N/2-1}|t_{p}|\bigg)= \frac{A_{s}}{h^{2s-1}}\big(-4\big(7f_{1}-4f_{2}+f_3\big)-2\mathcal G_{N/2-1}\big)
\\ &
= \frac{2A_{s}}{h^{2s-1}}\big(-2t_{1} -\mathcal G_{N/2-1}\big)> -\frac{4A_{s} t_{1}}{h^{2s-1}},
\end{split}
\end{equation}
where we used the fact $\mathcal G_{N/2-1}<0.$
We next show that there exists $N_0$ such that  the property \eqref{SPDD}  holds only for $ N\le N_0.$
Note from \eqref{gpa22} and \eqref{signhatcn} that the coefficients $\hat c_n^{(\alpha)}<0,$ so we have
$$
\mathcal G_{N/2-1}> \frac{2^{2s}\hat c_3^{(\alpha)}}{N^{2s} }, \quad \hat c_3^{(\alpha)}= (2s-1)(2s-2)(2s-3).
$$
It is clear that
\begin{equation}\label{GNN}
d_{N/2}>\frac{2A_{s}}{h^{2s-1}}\Big(-2t_{1} -  \frac{2^{2s}\hat c_3^{(\alpha)}}{N^{2s} } \Big):=\tilde d_{N/2}.
\end{equation}
We now search for the maximum possible $N$ so that $\tilde d_{N/2}>0,$ i.e.,
\begin{equation}\label{N0cond}
N^{2s}< -\frac{2^{2s-1}\hat c_3^{(\alpha)}}{t_1}   \quad {\rm or}\quad N\le N_0:=
\bigg[2\Big(\frac{(1-2s)(1-s)(3-2s) }{7+3^{3-2s}-2^{5-2s}}\Big)^{\frac 1 {2s}}\bigg],
\end{equation}
 for $s\in (0,s_0).$
  This completes the verification of Statement-(ii).  We  also refer to Figure \ref{t1sN0}(b) and Table \ref{TabN0} for the plot of $N_0=N_0(s)$  and some quantitative study.

\medskip
\underline{(iii).  $s\in (1,\frac 3 2)$:}\,  We now prove the last statement.
In this case, we have  $A_{s} < 0$ (cf. \eqref{AStiffMatrix0}), $t_0=2^{4-2s}-8<0, t_1>0$ and $t_p<0$ for $p\ge 2$ (cf. Lemma  \ref{elem1}), which implies
\begin{equation*}
S_{kk} = \frac{A_{s}t_{0}}{h^{2s-1}} > 0;\quad  S_{k,k\pm 1} = \frac{A_{s}t_{1}}{h^{2s-1}} < 0; \quad  S_{kj} = \frac{A_{s}t_{p}}{h^{2s-1}} > 0,\;\;\;  p=|k-j| \geq 2.
\end{equation*}
In what follows, we shall show that $d_1>0,$ but $d_2<0$ for all $N\ge 3.$
With this,  we can arrive at the conclusion in Statement-(iii) by using \eqref{symdk}-\eqref{monoticityA}.

We first show that $d_1>0.$  By (\ref{deltaT}) and \eqref{sjk01},
\begin{equation}\label{dsA1}
\begin{split}
d_1 & = \frac{A_{s}}{h^{2s-1}}\bigg(\!\!t_0 + t_1 - \sum_{p=2}^{N-2}t_p\bigg)
 = \frac{A_{s}}{h^{2s-1}}\big(2-5\cdot 2^{\alpha} + 2\cdot 3^{\alpha}+\mathcal G_{N-2}\big),\quad \alpha\in(0,1).
\end{split}
\end{equation}
Here by \eqref{signhatcn}, we have $\mathcal G_{N-2}<0$, so it suffices to show  $\tilde{g}(\alpha):=2-5\cdot 2^{\alpha} + 2\cdot 3^{\alpha}<0.$ Note that  $\tilde{g}'(\alpha) = -5(\ln2)2^{\alpha}+2(\ln3)3^{\alpha},$ which has a unique root
\begin{equation*}
\alpha_* = \frac{\ln(5\ln2)-\ln(2\ln3)}{\ln3-\ln2} \approx 1.124.
\end{equation*}
The function $\tilde{g}(\alpha)$ is descending for all $\alpha\in(0,1),$ so  $\tilde{g}(\alpha)<\tilde{g}(0)=-1$. Thus,  $d_1>0$ for all $N$.

We now show that $d_2<0$ for $N\ge 3.$ We obtain from \eqref{deltaT}  that
\begin{equation}
\begin{split}
d_2 & = \frac{A_{s}}{h^{2s-1}}\bigg(\!t_0 + 2t_1 - \sum_{p=2}^{N-3}t_p\bigg) = \frac{A_{s}}{h^{2s-1}}\big(3\big(3f_{1}-3f_{2}+f_3\big)+\mathcal G_{N-3}\big)
\\&= \frac{A_{s}}{h^{2s-1}}\big(-3\mathcal G_{1}+\mathcal G_{N-3}\big)
 < -\frac{3A_{s}}{h^{2s-1}}\mathcal G_{1},
\end{split}
\end{equation}
where we used the fact $\mathcal G_{1}, \mathcal G_{N-3}<0$.
Note from \eqref{gpa22} and \eqref{signhatcn} that the coefficients $\hat c_n^{(\alpha)}<0,$ so we have
$$
\mathcal G_{N-3}<\frac{\hat c_3^{(\alpha)}}{(N-2)^{2s} }, \quad \hat c_3^{(\alpha)}= (2s-1)(2s-2)(2s-3).
$$
If $-3\mathcal G_{1}+\mathcal G_{N-3}>0$, then it is clear that
\begin{equation*}
d_{2} < \tilde d_{2}:=\frac{A_{s}}{h^{2s-1}}\Big(-3\mathcal G_{1} + \frac{\hat c_3^{(\alpha)}}{(N-2)^{2s} } \Big)<0.
\end{equation*}
This yields
\begin{equation}
(N-2)^{2s}> \frac{\hat c_3^{(\alpha)}}{3\mathcal G_{1}}   \quad {\rm or}\quad N\ge N_1:=2+
\bigg[\Big(\frac{(2s-1)(2s-2)(2s-3) }{-9+9\cdot2^{3-2s}-3^{4-2s}}\Big)^{\frac 1 {2s}}\bigg].
\end{equation}
We can verify directly that we can take $N_1=3$ for all $s\in (1,\frac 3 2).$ This completes the proof.
%
%
%
\end{proof}

\begin{rem} {\em
From Lemma \ref{elem1}, we have that  $0<t_{1}(s)<2$ for $0<s<s_{0}$. Then let $r=\frac{1}{\cos(s_{0}\pi)\Gamma(4-2s_{0})}$, we can obtain a strictly diagonally dominant matrix for $(0,s_{0})$ by adding a diagonal matrix, i.e., $S_{h} = \bs S + 4r \bs I$ is strictly diagonally dominant for $s\in (0,1)$.  \qed}
\end{rem}

\begin{table}[!th]
\centering\small
\caption{Values of $N_0(s)$ and its asymptotic estimate  $N_a(s)$ for various $s\in (0, s_0)$} \label{TabN0}
\vspace*{-6pt}
\begin{tabular}{|c|c|c|c|c|c|c|c|c|c|c|c|}
\hline
\cline{1-1}
$s$   & $N_0(s)$ & $N_{a}(s)$ & $s$   & $N_0(s)$ & $N_{a}(s)$ &
$s$   & $N_0(s)$ & $N_{a}(s)$ & $s$   & $N_0(s)$ & $N_{a}(s)$ \\ \hline\hline
0.04  & 2573 & 2572  & 0.09  & 212  & 211  &
0.14  & 159  & 158   & 0.19  & 304  & 303\\\hline
0.05  & 986  & 985   & 0.10  & 184  & 183  &
0.15  & 166  & 165   & 0.20  & 419  & 418\\\hline
0.06  & 532  & 531   & 0.11  & 168  & 167  &
0.16  & 180  & 179   & 0.21  & 669  & 668\\\hline
0.07  & 350  & 349   & 0.12  & 159  & 158  &
0.17  & 204  & 203   & 0.22  & 1416 & 1415\\\hline
0.08  & 261  & 260   & 0.13  & 156  & 155  &
0.18  & 241  & 240   & 0.23  & 6728 & 6727\\\hline
\end{tabular}
\end{table}

At the end of this section, we illustrate the behaviour of the maximum/minimum eigenvalues of $\bs S$ for different fractional order $s\in (0,3/2).$ Observe from
Figure \ref{ProMatx}  that the maximum (resp. minimum) eigenvalue of the stiffness matrix $\bs S$ behaves  like $O(N^{2s-1})$ (resp. $O(N^{-1})$), so its condition number grows  like $O(N^{2s})$.
\begin{figure}[!th]
\centering
\subfigure[Minimum eigenvalue]{
\includegraphics[width=0.33\textwidth]{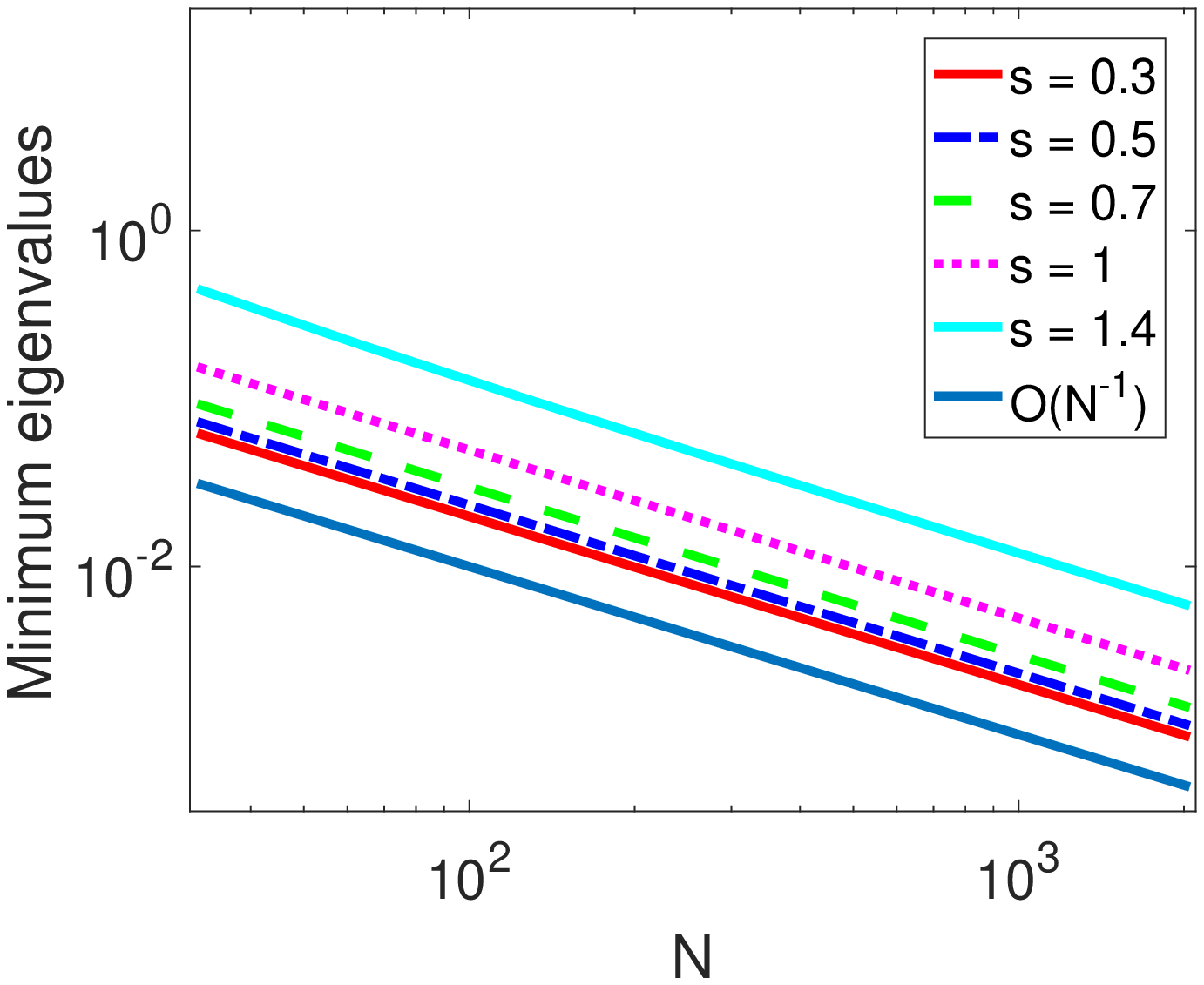}}\hspace*{-8pt}
\subfigure[Maximum eigenvalue]{
\includegraphics[width=0.33\textwidth]{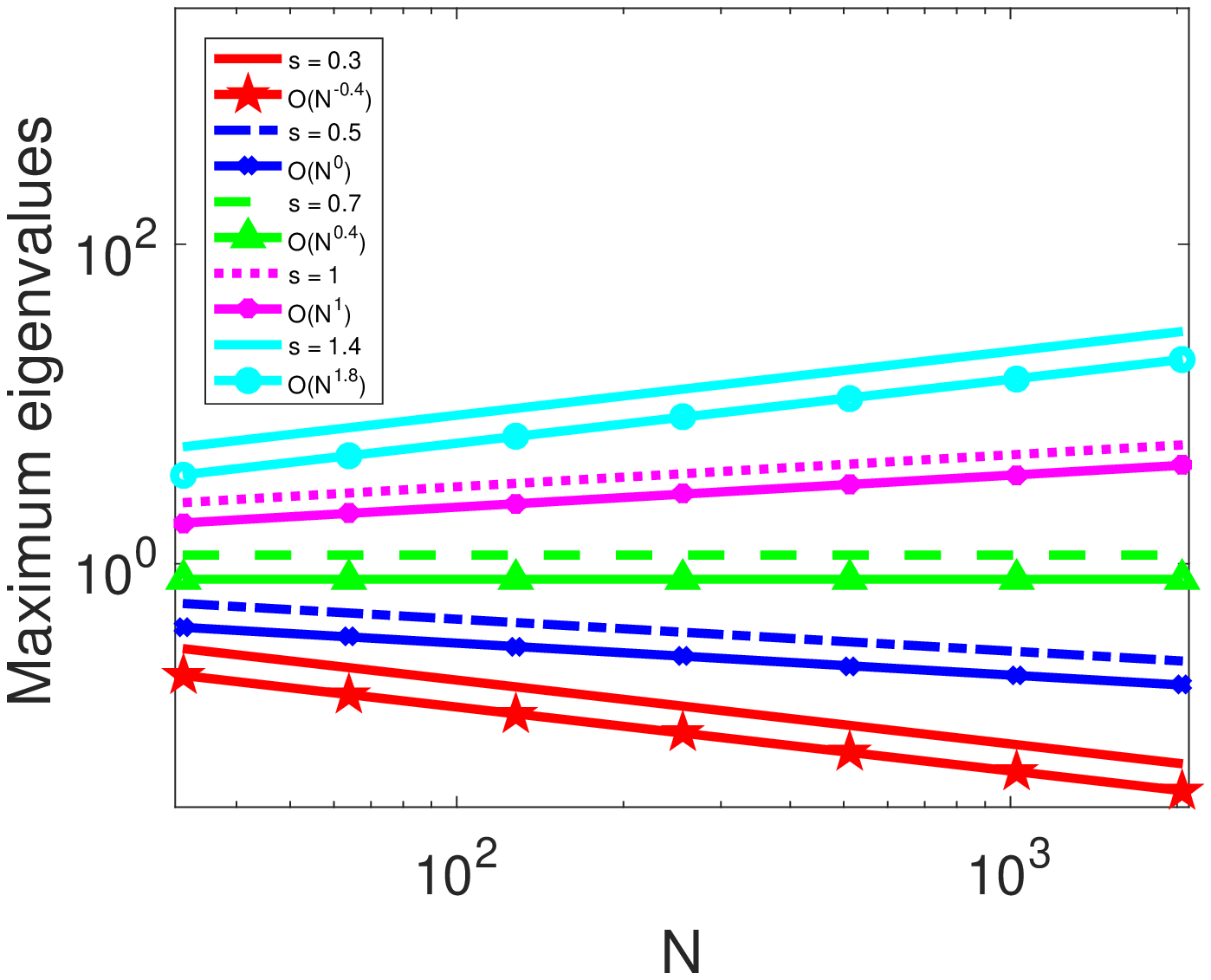}} \hspace*{-8pt}
\subfigure[Condition number]{
\includegraphics[width=0.33\textwidth]{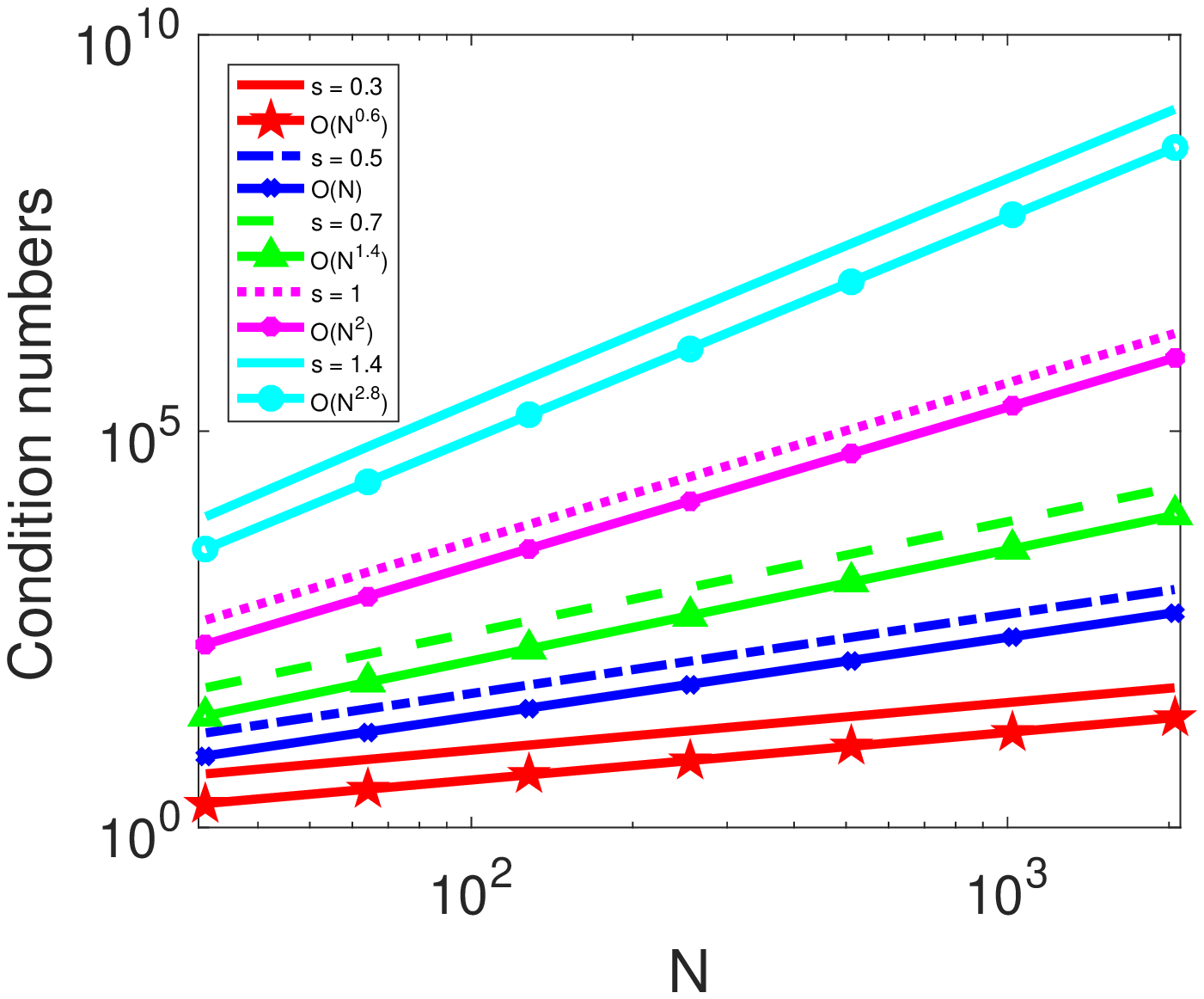}}
\caption{(a) The condition number of the stiffness matrix $\bs S$; (b) The maximum eigenvalue of the stiffness matrix $\bs S$; (c) The minimum eigenvalue of the stiffness matrix $\bs S$.}\label{ProMatx}
\end{figure}

\section{Maximum-principal  preserving schemes for  fractional-in-space Allen-Cahn equation}\label{sect3:AC}
\setcounter{equation}{0} \setcounter{lmm}{0} \setcounter{thm}{0}

In this section, we construct two  maximum principle preserving schemes for the fractional-in-space Allen-Cahn equation \eqref{fAC00}-\eqref{Fu00}, where the notion of  aforementioned  diagonal dominance plays an essential role.

\subsection{Maximum principle and energy dissipation of \eqref{fAC00}-\eqref{Fu00}}
We first show that in the fractional case, the maximum principle and energy dissipation law  hold  at continuous level.
In what follows, with a little abuse of notation,  we understand that $u=0$ for $x\in \Omega^c,$ when
the fractional Laplacian operator is performed on  $u(x).$
\begin{thm}
Let $u$ be the solution of the fractional Allen-Cahn equation \eqref{fAC00}-\eqref{Fu00}  with the fractional order $s\in (0, 3/2).$
If the initial value $u_0\in [0,1]$, then the solution $u\in [0,1]$ for all $t\in [0,T]$.
\end{thm}
\begin{proof} We first show that $u\leq 1$. 
For this purpose, we define
\begin{equation}\label{vu}
v:=\max\{u-1,0\}=(u-1)^{+}=
\begin{cases}
u-1, \ \ &u-1\geq 0,\\[2pt]
0, \ \  &u-1<0.
\end{cases}
\end{equation}
Taking the inner product of \eqref{fAC00} with $v$ yields that
\begin{equation}\label{maxprin1}
(u_{t},v)_{\Omega}+\epsilon^2 \big((-\Delta)^{s} u,v\big)_{\Omega}+\big(f(u),v\big)_{\Omega}=0.
\end{equation}
Then we can follow the proofs in \cite{feng2003numerical,li2015numerical} for usual Allen-Cahn equation to carry out the proof.  Firstly, we  can  show that
 \begin{equation}\label{utmultiplyu01}
(u_{t},v)_{\Omega}=\frac{1}{2}\frac{\rm d}{{\rm d} t}\int_{\Omega}|(u-1)^{+}|^2\, {\rm d}x;  \quad \big(f(u),v\big)_{\Omega}\geq 0,
\end{equation}
which follows immediately from
\begin{equation*}
u_{t}\,v=u_{t}\,(u-1)^{+}=(u-1)^+_{t}\,(u-1)^{+}=\frac{1}{2}\frac{\rm d}{{\rm d} t}|(u-1)^{+}|^2\,,
\end{equation*}
and
\begin{equation*}
f(u)\,v=\frac{u\,(2u-1)|(u-1)^{+}|^2}{2}\geq 0.
\end{equation*}

Then we prove the positiveness of the second term in \eqref{maxprin1}. Since for $x\in\Omega^{c}$, $u(x)=0$, i.e., $(u-1)^{+}=0$, we can write equivalently that
\begin{equation}\label{maxprin4}
\big((-\Delta)^s u,(u-1)^{+}\big)_{\Omega}=\big((-\Delta)^s u,(u-1)^{+}\big)_{\mathbb{R}}.
\end{equation}
For this integral in $\mathbb{R}$, we separate it into two parts: $R_{1}$, where $u(x)>1$, and $R_{2}$, where $u(x)\leq1$, i.e.,
\begin{equation}\label{maxprin3}
\big((-\Delta)^s u,(u-1)^{+}\big)_{\mathbb{R}}=\int_{R_{1}}\big((-\Delta)^s u\big)(u-1)^{+}{\rm d}x+\int_{R_{2}}\big((-\Delta)^s u\big)(u-1)^{+}{\rm d}x.
\end{equation}
For $x\in R_{1}$, we know $(u(x)-1)^{+}=u(x)-1$. Then by the definition of fractional Laplacian in \eqref{fracLap-defn} and the fact that $(u(y)-1)^{+}\geq u(y)-1$, $\forall y\in\mathbb{R}$, one verifies readily that
\begin{equation*}
(-\Delta)^s u(x) =(-\Delta)^s \big(u(x)-1\big)\geq(-\Delta)^s \big(u(x)-1\big)^{+}, \;\;\;  x\in R_{1}.
\end{equation*}
Then since $(u(x)-1)^{+}\geq 0$, $\forall x\in \mathbb{R}$, we have for the first integral of \eqref{maxprin3} that
\begin{equation}\label{ubig1}
\int_{R_{1}}\big((-\Delta)^s u\big)(u-1)^{+}{\rm d}x\geq \int_{R_{1}}\big((-\Delta)^s (u-1)^{+}\big)(u-1)^{+}{\rm d}x.
\end{equation}
On the other hand, for $x\in R_{2}$, $(u(x)-1)^{+}=0$, thus for the second integral of \eqref{maxprin3}
\begin{equation}\label{usmall1}
\int_{R_{2}}\big((-\Delta)^s u\big)(u-1)^{+}{\rm d}x=0=\int_{R_{2}}\big((-\Delta)^s (u-1)^{+}\big)(u-1)^{+}{\rm d}x.
\end{equation}
Then finally combining \eqref{maxprin4}-\eqref{usmall1}, we have
\begin{equation}\label{maxprin2}
\big((-\Delta)^s u,(u-1)^{+}\big)_{\Omega}\!\geq\big((-\Delta)^s (u-1)^{+},(u-1)^{+}\big)_{\mathbb{R}}\!=\big((-\Delta)^{s/2} (u-1)^{+},(-\Delta)^{s/2}(u-1)^{+}\big)_{\mathbb{R}}\!\geq 0.
\end{equation}

From \eqref{maxprin1}, \eqref{utmultiplyu01} and \eqref{maxprin2}, we derive
\begin{equation*}
\frac{\rm d}{{\rm d} t}\int_{\Omega}|(u-1)^{+}|^2\, {\rm d}x\le 0\;\; {\rm i.e.,}\;\;
\int_{\Omega}|(u-1)^{+}|^2{\rm d} x\le \int_{\Omega}|(u_0-1)^{+}|^2{\rm d} x=0,
\end{equation*}
as $u_{0}\leq 1.$ This implies $u\leq 1$.


\smallskip
Similarly, we define the test function 
\begin{equation*}
w:=u^{-}=
\begin{cases}
-u, \ \ &u\leq 0,\\
0, \ \ &u>0.
\end{cases}
\end{equation*}
Following the same lines,  we can prove $u\geq 0$. Since $u=0$ for $x\in \Omega^{c}$, we have $0\leq u\leq 1$. This completes the proof.
\end{proof}

The energy dissipation law is easy to show. Define the energy as in \cite{duo2019fractional,tang2019energy}
\begin{equation}\label{EnergyEu}
E(u) = \int_{\Omega} \Big\{\frac{\epsilon^2}{2}\big((-\Delta)^{s/2} u\big)^2 + F(u)\Big\} {\rm d}x.
\end{equation}
Taking the inner product of \eqref{fAC00} with $u_{t}$ on $\Omega$, we have
\begin{equation*}
\int_{\Omega}|u_{t}|^2{\rm d}x+\epsilon^2 \int_{\Omega}\big\{(-\Delta)^s u\,u_{t}+f(u)u_{t}\big\}{\rm d}x=0.
\end{equation*}
As with the standard Allen-Cahn equation in bounded domain, we have
\begin{equation*}
\frac{{\rm d}}{{\rm d}t} E(u)=-\|u_{t}\|^2\leq 0.
\end{equation*}
\begin{rem}\label{multiDcase} {\em The above two properties also hold in multiple dimensions. \qed}
\end{rem}

In the next two subsections, we propose two full-discrete schemes, which can preserve these two properties,  and are of the first and second-order accuracy in time respectively.
\subsection{ Standard semi-implicit  time discretization with modified FEM in space}\label{fies}
We first define the  piecewise linear interpolation $I_{h}: C(\bar{\Omega})\mapsto \V_{\!h}$ as
\begin{equation}\label{Ih}
I_{h}v(x):=\sum_{j=1}^{N-1}v(x_{j})\phi_{j}(x)=\sum_{j=1}^{N-1}v_{j}\phi_{j}(x).
\end{equation}
Then the full-discrete  scheme is to find $u_{h}^{n+1}\in \V_{\!h}$ such that
\begin{equation}\label{modified}
\frac{1}{\tau}\big(I_{h}\big((u_{h}^{n+1}-u_{h}^{n})v_h\big),1\big)_\Omega+\epsilon^2 \big((-\Delta)^s u_{h}^{n+1},v_{h}\big)_\Omega+\big(I_{h}\big(f(u_{h}^n)v_{h}\big),1\big)_\Omega=0,\;\; \forall v_{h}\in \V_{\!h}.
\end{equation}
Its matrix form reads
\begin{equation}\label{modified1}
\frac{U^{n+1}-U^{n}}{\tau}+\frac{\epsilon^2}{h} {\bs S}\, U^{n+1}+f(U^n)=0,
\end{equation}
where $\bs S$ is the stiffness matrix in \eqref{StiffMatrix}, and
\begin{equation*}
U^{n}=\big(u_h^{n}(x_1),\cdots, u_h^{n}(x_{N-1})\big)^T,\quad
f(U^{n}) = \big(f(u_{h}^{n}(x_1)),\cdots, f(u_{h}^{n}(x_{N-1}))\big)^{T}.
\end{equation*}
%
\begin{rem}{\em Different from the usual finite element discretisation, we adopted the modification as in the recent work by Xu
et al. \cite{xu2019stability} in the study of time-discretisation of the standard  Allen-Cahn equation. \qed
}
\end{rem}

\subsubsection{Maximum principle of the scheme \eqref{modified1}} To prove the maximum principle, we first show the following important properties   of  $\bs S,$ drawn from Theorem
\ref{DM01}.

\begin{lmm}\label{samesign}
For $s\in(s_{0},1]$ and a given nonzero vector $v=(v_1, \cdots, v_{N-1})^T$, the stiffness matrix $\bs S$  in Theorem \ref{th21} has the following properties:
\begin{itemize}
\item[(i)] Suppose that  $v$ has at least one negative component, and let $v_p$ be the component with the biggest absolute value among the negative components. Then we have  $\sum_{j=1}^{N-1}{S}_{pj}v_{j}<0$.
\medskip
\item[(ii)]Suppose that  $v$ has at least one positive component, and let $v_q$ be the component with the largest value. Then we have $\sum_{j=1}^{N-1}{S}_{qj}v_{j}> 0$.
\end{itemize}
\end{lmm}
\begin{proof} From Theorem \ref{DM01}, we know that for $s\in(s_{0},1),$ the matrix $\bs S$ is diagonally dominant such that
for $1\le i\le N-1,$
\begin{equation*}
{S}_{ii} > 0; \;\;\;  {S}_{ij}<0,\;\;\; \text{if}\;i\neq j;  \;\;\;  {\rm and} \;\;\;  {S}_{ii} + \sum_{j\neq i}{S}_{ij} >0.
\end{equation*}
We first prove  Statement (i). Since $v_{p}<0$, it is equivalent to proving that $\sum_{j=1}^{N-1}{S}_{pj}v_j$ and $v_{p}$ have the same sign.  It is evident that
\begin{equation*}
v_{p}\Big(\sum_{j=1}^{N}{S}_{pj}v_{j}\Big)\!=v_{p}\Big(\!{S}_{pp}v_{p}+\sum_{j\neq p}{S}_{pj}v_{j}\Big)\!={S}_{pp}v_{p}^{2}+\sum_{j\neq p}{S}_{pj}v_{p} v_{j}.
\end{equation*}
Let   $J=\{j : v_{j}<0,\; 1\le j\le N-1\}.$ Since ${S}_{ii} > 0,\; {S}_{ij}<0\;\text{for}\;i\neq j,$ we find
\begin{equation*}
\begin{split}
v_{p}\Big(\sum_{j=1}^{N-1}{S}_{pj}v_{j}\Big) &\geq {S}_{pp}v_{p}^{2}+\!\sum_{j\in J,\,j\neq p}\!{S}_{pj}v_{p} v_{j}= {S}_{pp}v_{p}^{2}+\!\sum_{j\in J,\,j\neq p}\!{S}_{pj}|v_{j}|\,|v_{p}|\\
&= \Big({S}_{pp}+ \sum_{j\in J,\,j\neq p}\!\!{S}_{pj}\Big)|v_{p}|^2> 0,
\end{split}
\end{equation*}
where in the last step, we used the fact:  $|v_{p}|=\max\{|v_{j}|,\,j\in J\}$.

The second statement can be proved in the same fashion. It is evident that the above properties hold for $s=1.$
\end{proof}

We are now in a position to show that the scheme \eqref{modified1} preserves the maximum principle.
\begin{thm}\label{FDMEh} For $s\in (s_0, 1],$
if the initial value satisfies $0\leq u_{0}(x) \leq 1$, then the full-discrete scheme \eqref{modified1} preserves the maximum principle in the sense that $0\leq U^{n}_j \leq 1$ for $1\le j\le N-1,$ if the  time stepping size $0< \tau \leq 2$.
\end{thm}
\begin{proof} We rewrite the scheme  \eqref{modified1} as
\begin{equation}\label{differencelikescheme}
U^{n+1}+\frac{\tau\epsilon^2}{h}{\bs S}\,U^{n+1}=U^{n}-\tau f(U^n).
\end{equation}
We carry out the proof by mathematical induction. It's obvious that $0\leq U^{0}_j\leq 1$. Assuming  that $0\leq U^{n}_j \leq 1$, we next show that  $0\leq U^{n+1}_j \leq 1$. Observe that the component of $U^{n}-\tau f(U^{n})$ is associated with
\begin{equation*}
g(x)=x-\tau f(x).
\end{equation*}
One verifies that if $0<\tau\leq 2$, then for $x\in [0,1],$
\begin{equation*}
g'(x)=-3\tau\Big(x-\frac{1}{2}\Big)^{2}+1+\frac{\tau}{4}\geq 1-\frac{\tau}{2}\geq 0.
\end{equation*}
Thus
\begin{equation*}
\min_{x\in[0,1]}g_{1}(x)=g_{1}(0)=0,\quad\max_{x\in[0,1]}g_{1}(x)=g_{1}(1)=1.
\end{equation*}
Since $0\leq U^{n}_j \leq 1$, we have
\begin{equation}\label{forg2}
0\leq U^{n}_j-\tau f(U^{n}_j)\leq 1,\;\;\; \text{if}\;\; 0<\tau\le 2.
\end{equation}
 We proceed with the proof by contradiction.  
 If there exists a negative component in $U^{n+1}$, we choose the one with the biggest absolute value, say $U_{p}^{n+1}$. Then by Lemma \ref{samesign}, we have $\sum_{j=1}^{N-1}S_{pj}U_{j}^{n+1}< 0,$ so
\begin{equation*}
U_{p}^{n+1}+\frac{\tau\epsilon^2}{h}\sum_{j=1}^{N-1}S_{pj}U_{j}^{n+1}< 0,
\end{equation*}
which contradicts to
 the $p$th equation of the system \eqref{differencelikescheme},  in view of \eqref{forg2}.  Thus all components  $U^{n+1}_j\geq 0$.

 On the other hand, if there exists a component in $U^{n+1}$ that is bigger than $1$, we choose the one with the biggest value, say $U_{q}^{n+1}$. Then by Lemma \ref{samesign}, we have $\sum_{j=1}^{N-1}S_{qj}\,U_{j}^{n+1}>0,$  which implies
\begin{equation*}
U_{q}^{n+1}+\frac{\tau\epsilon^2}{h}\sum_{j=1}^{N-1}S_{qj}U_{j}^{n+1}> 1.
\end{equation*}
Once again, by virtue of \eqref{forg2}, we find it contradicts to the $q$th equation of the system \eqref{differencelikescheme}.
Thus we have $U^{n+1}_j\leq 1$.
This completes the proof.
\end{proof}
\subsubsection{Energy dissipation of the scheme \eqref{modified1}}
Corresponding to \eqref{EnergyEu}, we define the discrete energy  as 
\begin{equation}\label{Ehu}
E_{h}(U) =  \frac{\epsilon^2}{2}U^T{\bs S}U + h\sum_{j=1}^{N-1}F(U_j),
\end{equation}
where $\bs S$ is defined in Theorem \ref{th21} and $U$ is a column vector of length $N-1$.  Then we have the following energy dissipation property.
\begin{thm}\label{energydis} Let $s\in (s_0,1].$
If the initial value satisfies $0\leq u_{0}(x)\leq 1$, then the numerical solutions of the scheme \eqref{modified1} satisfies the discrete energy dissipation law:
\begin{equation}\label{EhU}
E_{h}(U^{n+1})\leq E_{h}(U^{n}),
\end{equation}
if $0< \tau \leq 2$.
\end{thm}
The proof is very similar to the proof of   \cite[Theorem 2.2]{tang2016implicit}, so we omit it.
\begin{rem}{\em
It is proved above that when $\tau\in (0,2]$, the maximum principle and energy dissipation are maintained in \eqref{modified1}. In order to make it unconditionally stable, we can add an extra perturbation term which is compatible with the truncation error, see \cite{tang2016implicit} for more details. \qed}
\end{rem}

\subsection{Modified FEM and Crank-Nicolson scheme}\label{cns}
Now we present the modified FEM and Crank-Nicolson scheme:
\begin{equation}\label{modifiedCN}
\begin{split}
\frac{1}{\tau}\big(I_{h}\big((u_{h}^{n+1}-u_{h}^{n})v_h\big),1\big)_{\Omega}+&\epsilon^2 \Big(\frac{(-\Delta)^s u_{h}^{n+1}+(-\Delta)^s u_{h}^{n}}{2},v_{h}\Big)_\Omega
\\&+\Big(I_{h}\Big(\frac{f(u_{h}^{n+1})+f(u_{h}^n)}{2}v_{h}\Big),1\Big)_\Omega=0,\quad\,\forall v_{h}\in \V_{\!h},
\end{split}
\end{equation}
with the matrix form
\begin{equation}\label{modified1CN}
\frac{ U^{n+1}-U^{n}}{\tau}+\frac{\epsilon^2}{2h} {\bs S} \big(U^{n+1}+U^{n}\big)+\frac{1}{2}\big(f(U^{n+1})+f(U^n)\big)=0.
\end{equation}

\subsubsection{Maximum principle of the scheme \eqref{modifiedCN}}
\begin{theorem}
If the initial value satisfies $0\leq u_{0}(x)\leq 1$, then for $s\in(s_{0},1]$, the scheme \eqref{modifiedCN} preserves the maximum principle in the sense that $0\leq U_{j}^{n}\leq 1$ for all $n\geq 1$ and $j=1,2,\cdots,N$, provided that the time stepsize satisfies
\begin{equation}\label{stepCN}
0<\tau\leq \min\Big\{2,\frac{h^{2s}}{2\epsilon^{2}}\Big\}.
\end{equation}
\end{theorem}
\begin{proof}
Rewrite \eqref{modified1CN} as
\begin{equation}\label{modified2CN}
\Big(\frac{\bs I}{2}+\frac{\tau\epsilon^2}{2h} {\bs S}\Big)U^{n+1}+\frac{U^{n+1}+\tau f(U^{n+1})}{2}=\Big(\frac{\bs I}{2}-\frac{\tau\epsilon^2}{2h} {\bs S}\Big)U^{n}+\frac{U^{n}-\tau f(U^{n})}{2}.
\end{equation}
Denote its right hand side by
\begin{equation}\label{RHS}
\text{RHS}:=\Big(\frac{\bs I}{2}-\frac{\tau\epsilon^2}{2h} {\bs S}\Big)U^{n}+\frac{U^{n}-\tau f(U^{n})}{2}.
\end{equation}
Let $\bs H=\frac{\bs I}{2}-\frac{\tau\epsilon^2}{2h} {\bs S}$. Then for  $s\in(s_{0},1]$,  the matrix $\bs H$ satisfies
\begin{equation}\label{iicond}
{\rm (i)} \;\;\;h_{ii}=\frac{1}{2}-\frac{\tau\epsilon^2}{2h^{2s}\varphi(s)};\quad {\rm (ii)} \;\;\; h_{ij}\big|_{j\neq i}\geq 0\;\; \text{and}\;\; \max_{i}\sum_{j}h_{ij}\leq \frac{1}{2},
\end{equation}
where  
\begin{equation*}
\varphi(s)=
\begin{cases}
\dfrac{\cos(s\pi)\Gamma(4-2s)}{2^{3-2s}-4},\;& s\in(s_{0},\frac{1}{2})\cup(\frac{1}{2},1],\\[6pt]
\dfrac{\pi}{16\ln2},\qquad\quad\;\;\, &s=\frac{1}{2}.
\end{cases}
\end{equation*}
It is evident that if $\tau\leq\frac{h^{2s}}{\epsilon^{2}}\varphi(s),$ then $h_{ii}\geq 0,$ so  the elements of $\bs H$ are nonnegative,
in view of \eqref{iicond}.
Consequently,
\begin{equation}\label{forH}
\|\bs H\|_{\infty}=\max_{i}\sum_{j}|h_{ij}|=\max_{i}\sum_{j}h_{ij}\leq \frac{1}{2}.
\end{equation}
One verifies readily that for $s\in (s_0,1]$, $\varphi(s)\ge \frac 1 2$.
Thus, we can simplify the constraint imposed on time step  to
$\tau\leq\frac{h^{2s}}{2\epsilon^{2}}. $

For the last term in \eqref{RHS}, we know from the proof of Theorem \ref{FDMEh} that
\begin{equation}\label{forg}
0\leq U^{n}-\tau f(U^{n})\leq 1, \;\;\; \text{for}\; U^{n}\in[0,1], \;\; \text{if}\;  \tau<2.
\end{equation}
Thus, if $U^{n}\in[0,1]$ and $0<\tau<\min\big\{\frac{h^{2s}}{2\epsilon^{2}},2\big\}$, we obtain from  \eqref{forH}-\eqref{forg} that
\begin{equation*}
0\leq\text{RHS}_{i}=\sum_{j}h_{ij}U_{j}^{n}+\frac{U_{i}^{n}-\tau f(U_{i}^{n})}{2}\leq\sum_{j}h_{ij}+\frac{1}{2}\leq 1.
\end{equation*}

Next, denote the left hand side of \eqref{modified2CN} as
\begin{equation}\label{LHS}
\text{LHS}:=\Big(\frac{\bs I}{2}+\frac{\tau\epsilon^2}{2h} {\bs S}\Big)U^{n+1}+\frac{U^{n+1}+\tau f(U^{n+1})}{2}.
\end{equation}
 We  first consider the last term in \eqref{LHS}. Observe that each element of $U^{n+1}+\tau f(U^{n+1})$ is of the form
\begin{equation}\label{g2}
g_{2}(x)=x+\tau f(x).
\end{equation}
It can be verified that if $\tau\leq 2$,
\begin{equation*}
g_{2}'(x)=3\tau(x-\frac{1}{2})^{2}+1-\frac{\tau}{4}\geq 1-\frac{\tau}{4}\geq 0,\;\,\text{for}\;\,\tau<2.
\end{equation*}
Since $g_{2}(0)=0$, we have
\begin{equation*}
g_{2}(x)>0,\;\text{when}\;x>0;\quad g_{2}(x)<0,\;\text{when}\;x<0.
\end{equation*}
Thus, we have $U^{n+1}$ and $U^{n+1}+\tau f(U^{n+1})$ are positive or negative simultaneously.

Then we can prove $0\leq U^{n+1}\leq 1$ by contradiction using Lemma \ref{samesign} as in Theorem \ref{FDMEh}, we omit it here. 
\end{proof}
\subsubsection{Energy dissipation of the scheme \eqref{modifiedCN}} 
\begin{theorem}
If the initial value satisfies $0\leq u_{0}(x)\leq 1$, then for $s\in(s_{0},1]$, the numerical solutions of the scheme \eqref{modified1} satisfies the discrete energy dissipation law:
\begin{equation}
E_{h}(u^{n+1})\leq E_{h}(U^{n}),
\end{equation}
under the time step constraint in \eqref{stepCN}.
\end{theorem}\begin{proof}
The proof is quite similar to that of   \cite[Theorem 2]{hou2017numerical}, so we omit it.
\end{proof}

\section{Numerical results}\label{sect4:numer}
\setcounter{equation}{0} \setcounter{lmm}{0} \setcounter{thm}{0}
In the section, we provide ample numerical results for the model \eqref{fAC00} by using the proposed schemes with the focus on the illustration of the preservation of maximum principle and energy dissipation.  We also explore the dynamics of the fractional model and show its transition the usual Allen-Cahn model, e.g., the width of the interface in terms of $s.$

\subsection{Accuracy test}  We start with testing the accuracy of the FEM scheme with the stiffness matrix $\bs S$ given in  Theorem \ref{th21}
for  solving the fractional Poisson equation:
\begin{equation}\label{tests032}
(-\Delta)^s u(x) = f(x),\;\;\; x\in\Omega:=(-1,1); \quad
u(x) = 0,\;\;\;  x\in \Omega^c=\mathbb{R}\backslash\bar\Omega,
\end{equation}
for $s\in (0,3/2),$ which admits the exact solution: $u(x)=(1-x^2)^{n+s}$ in $\Omega$ and $u(x)=0$ on $\Omega^c$ for  integer $n\ge 1,$ if (cf. \cite{MR698779})
\begin{equation}\label{fformA}
f(x) = \frac{2^{2s}\Gamma(s+\frac12)\Gamma(n+1+s)}{\Gamma(n+\frac12)}P^{(-\frac12,s)}_n(1-2x^2),\ \ x\in \Omega,
\end{equation}
where $P^{(\alpha,\beta)}_n(x)$ denote the $n$-th Jacobi polynomials (cf. \cite{ShenTangWang2011}).

In Tables \ref{MerrOrdernk}, we tabulate the maximum errors and the corresponding convergence rates (c.r.) obtained by the proposed method with various $s\in(0,3/2)$, for which we take $n=1$ and $n=3$. We observe that the numerical errors decay as $h$ decreases for any fixed $s$. As we all know, the convergence rate under $L^\infty$-norm of piecewise linear approximation for the singular function $(1-x^2)^{n+s}$ is $O(h^{\min(2,n+s)})$ in finite element analysis (see, e.g., \cite{babuska2001finite}). In particular, they indicate that the convergence rates of our method is $O(h^{1+s})$ when $n=1$ (see  the left side of Table \ref{MerrOrdernk}) and is $O(h^{2})$ when $n=3$ (see right side of Table \ref{MerrOrdernk}) for any $s\in(0,1)$, which confirm the theoretical expectations.

\begin{table}[!th]
\caption{Errors and convergence rates.}
\centering
\setlength{\tabcolsep}{1mm}{
\begin{tabular}{c c c c c c c c c c c}
\hline
\multirow{2}{*}{$(s,h)$}
& \multicolumn{5}{c}{$n=1$}
& \multicolumn{5}{c}{$n=3$}\\
\cmidrule(lr){2-6}  \cmidrule(lr){7-11}
& $2^{-5}$ & $2^{-6}$ & $2^{-7}$ & $2^{-8}$ & $2^{-9}$
& $2^{-5}$ & $2^{-6}$ & $2^{-7}$ & $2^{-8}$ & $2^{-9}$ \\
\hline
0.3  & 9.04e-4  & 4.27e-4 & 1.86e-4
     & 7.80e-5  & 3.22e-5
     & 1.07e-3  & 2.69e-4 & 6.71e-5
     & 1.68e-5  & 4.22e-6\\
c.r. & --       & 1.08    & 1.20
     & 1.25     & 1.27
     & --       & 2.00    & 2.00
     & 2.00     & 1.99 \\
\hline
0.5  & 8.26e-4  & 2.76e-4 & 1.05e-4
     & 3.82e-5  & 1.37e-5
     & 1.13e-3  & 2.84e-4 & 7.11e-5
     & 1.78e-5  & 4.45e-6\\
c.r. & --       & 1.58    & 1.40
     & 1.45     & 1.48
     & --       & 2.00    & 2.00
     & 2.00     & 2.00\\
\hline
0.95 & 1.14e-3  & 2.93e-4 & 7.47e-5
     & 1.90e-5  & 4.88e-6
     & 8.86e-4  & 2.28e-4 & 5.86e-5
     & 1.50e-5  & 3.87e-6\\
c.r. & --       & 1.97    & 1.97
     & 1.97     & 1.96
     & --       & 1.96    & 1.96
     & 1.96     & 1.96\\
\hline
1    & 9.77e-4  & 2.44e-4 & 6.10e-5
     & 1.53e-5  & 3.81e-6
     & 6.50e-4  & 1.63e-4 & 4.07e-5
     & 1.02e-5  & 2.54e-6\\
c.r. & --       & 2.00    & 2.00
     & 2.00     & 2.00
     & --       & 2.00    & 2.00
     & 2.00     & 2.00\\
\hline
1.2  & 3.33e-3  & 1.27e-3 & 4.59e-4
     & 1.61e-4  & 5.58e-5
     & 3.80e-3  & 1.36e-3 & 4.78e-4
     & 1.64e-4  & 5.60e-5\\
c.r. & --       & 1.40    & 1.47
     & 1.51     & 1.53
     & --       & 1.48    & 1.51
     & 1.54     & 1.55\\
\hline
\end{tabular}}\label{MerrOrdernk}
\end{table}

Next, we test the convergence rate of two full discrete schemes for \eqref{fAC00}, i.e., semi-implicit and Crank Nicolson scheme established in the last section.
For this purpose, we consider \eqref{fAC00} with an exact solution by adding an extra right hand side $g(x,t)$, that is,
\begin{equation*}
u_{t}(x,t)+\epsilon^2 (-\Delta)^s u(x,t)+f(u(x,t))=g(x,t),\quad x\in [-L,L],\;t\in(0,T],
\end{equation*}
where the boundary and initial condition is given in \eqref{fAC00}.
We take $u(x,t)=e^{-t-\lambda^2x^2}$, then the decay rate of $u(x,t)$ can be easily controlled by choosing different $\lambda$.
According to \cite[Prop.\,4.2]{sheng2019fast}, we find that the right side function can be expressed as
\begin{equation*}
\begin{split}
g(x,t)= &\pi^{-\frac12}(2\lambda)^{2s}\Gamma(s+\frac12)\,\epsilon^2\, e^{-t}\,{}_{1}F_{1}\big(s+\frac{1}{2};\frac{1}{2};-\lambda^2x^2\big)
\!-\frac{1}{2}e^{-t-\lambda^2x^2}\!-\frac{3}{2}e^{-2t-2\lambda^2x^2}\!+e^{-3t-3\lambda^2x^2},
\end{split}
\end{equation*}
where the confluent hypergeometric function (cf.\,\cite{Nist2010})
\begin{equation}\label{1f1}
{}_{1}F_{1}(a;b;x)=\sum_{n=0}^\infty\frac{(a)_n\,x^n}{(b)_n\,n!},
\end{equation}
with $(\cdot)_n$ denote the rising factorial in the Pochhammer symbol.

We take $\epsilon=0.1$, $L=1$, $T=1.6$, $\lambda=10$, and $s=0.8$. On the left side of Table \ref{convergencerate1}, we present the spatial error for both semi-implicit and Crank Nicolson schemes, for which we fix time step $\tau=10^{-5}$ so that the temporal error is negligible. To test the temporal accuracy, we choose the mesh size $h=2^{-11}$ to make sure the temporal error dominates the error, and list temporal errors for both semi-implicit and Crank Nicolson schemes on the right side of Table \ref{convergencerate1}. We observe that spatial error is $O(h^2)$, while the temporal errors are $O(\tau)$ and $O(\tau^2)$ for semi-implicit and Crank Nicolson scheme, respectively.

\begin{table}[!ht]
\caption{Errors and convergence rates with respect to $(h,\tau)$.}
\begin{center}
\setlength{\tabcolsep}{2.5mm}{
\begin{tabular}{c c c c c c c c c c}
\hline
\multirow{2}{*}{}
& \multicolumn{2}{c}{semi-implicit }
& \multicolumn{2}{c}{CN}
& \multirow{2}{*}{}
& \multicolumn{2}{c}{semi-implicit }
& \multicolumn{2}{c}{CN}\\
\cmidrule(r){2-3} \cmidrule(r){4-5}
\cmidrule(r){7-8} \cmidrule(r){9-10}
$h$&  Errors   &  c.r. & Errors & c.r.
&
$\tau$&  Errors   &  c.r. & Errors & c.r.\\
\hline
$2^{-5}$ & 9.38e-3 & --   & 9.37e-3 & -- &
$1/5$    & 6.87e-2 & --   & 1.92e-3 & -- \\
\hline
$2^{-6}$ & 2.37e-3 & 1.98 & 2.37e-3 & 1.98 &
$1/10$   & 3.47e-2 & 0.98 & 4.77e-4 & 2.00\\
\hline
$2^{-7}$ & 6.11e-4 & 1.96 & 6.08e-4 & 1.96 &
$1/20$   & 1.74e-2 & 0.99 & 1.17e-4 & 2.02\\
\hline
$2^{-8}$ & 1.59e-4 & 1.94 & 1.56e-4 & 1.97 &
$1/40$   & 8.72e-3 & 1.00 & 2.74e-5 & 2.10\\
\hline
$2^{-9}$ & 4.31e-5 & 1.88 & 3.96e-5 & 1.97 &
$1/80$   & 4.37e-3 & 1.00 & 5.29e-6 & 2.38\\
\hline
\end{tabular}}
\end{center}
\label{convergencerate1}
\end{table}

%
%
%
%
\subsection{Fractional-in-space Allen-Cahn equation}  
In this subsection, we focus on the simulation of phase evolution behavior and interfacial behavior for fractional Allen-Cahn equations \eqref{fAC00}. In what follows, we restrict our attention to the semi-implicit scheme \eqref{modified}, as the results obtained by the Crank Nicolson scheme \eqref{modifiedCN} are very similar.  Both schemes satisfy the maximum principle, and the only difference is the convergence rate, i.e., first order and second order in time.
\subsubsection{Phase separation}
We take $\Omega=(-2,2),$ $\epsilon=0.01$, $\tau=0.01,$ $h=2^{-10}$ and the initial data $u_{0}(x)=\frac45\,e^{-x^2}$. In Figure \ref{profile}, we present the snapshots of the solutions at $t=0,4,8,12$ with $s=0.3$ and $s=0.7$. As with classical Allen-Cahn equations (i.e., $s=1$), the phase separation phenomenon is observed, where the solutions gradually correspond to the minimizer of the total energy as time goes on.
\begin{figure}[!ht]
\centering
\subfigure[$s=0.3$]{
\includegraphics[width=0.45\textwidth]{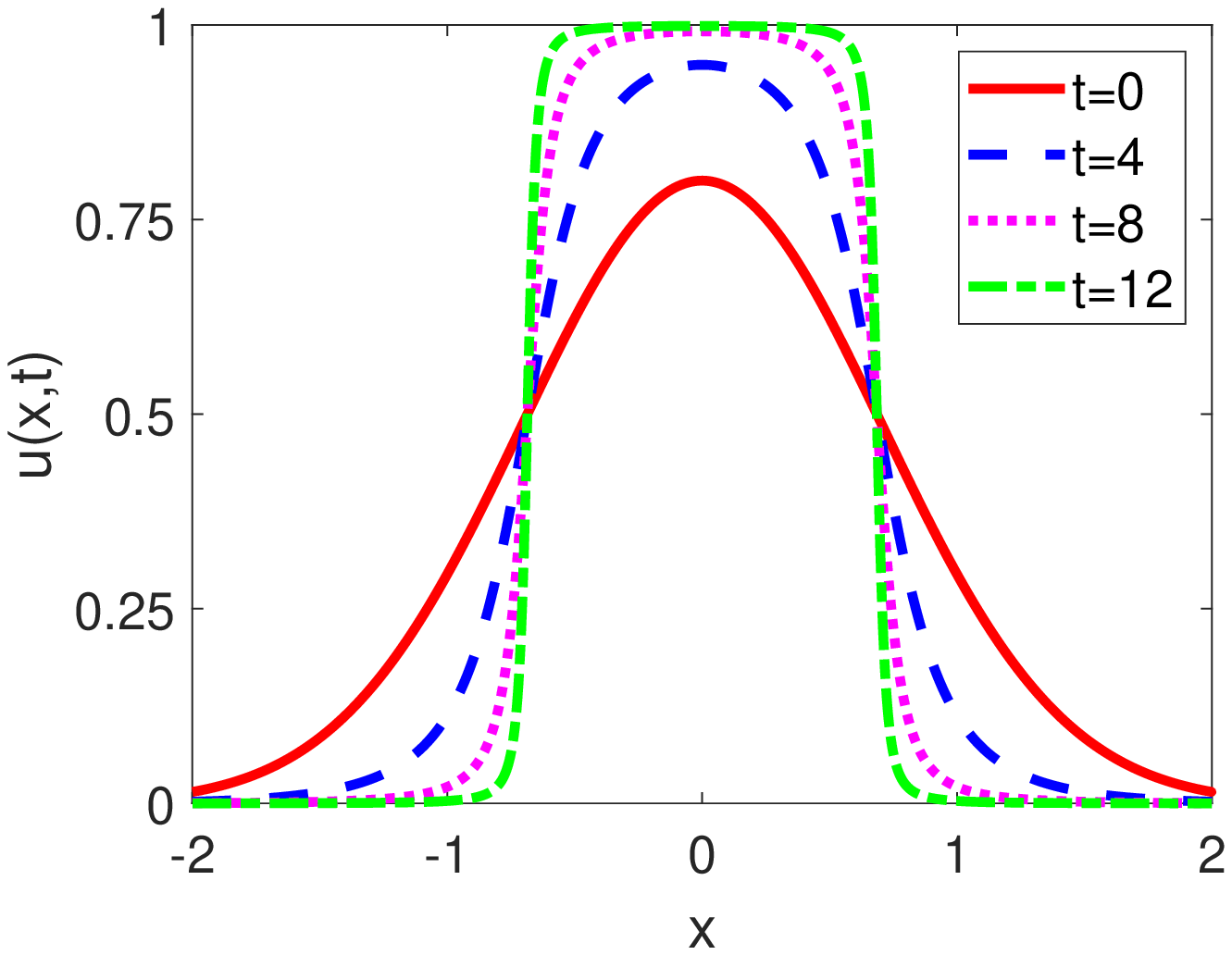}}
\subfigure[$s=0.7$]{
\includegraphics[width=0.45\textwidth]{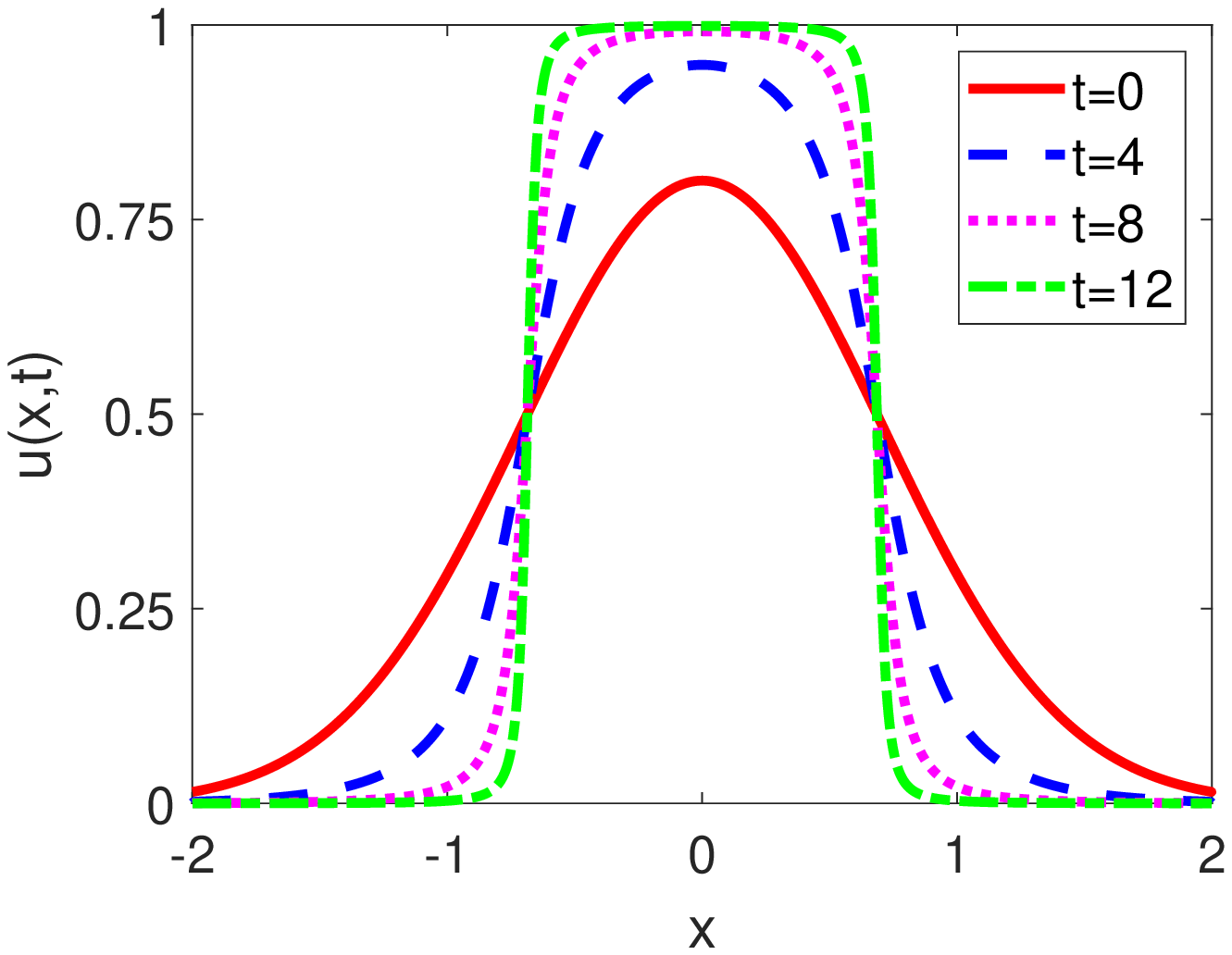}}\vspace{-8pt}
\caption{Profiles of solution to the fractional Allen-Cahn equation with the initial condition $u_{0}(x)=\frac45\,e^{-x^2}$. Snapshots of $u(x)$ are taken at $t=0,4,8,12$. }
\label{profile}
\end{figure}


\subsubsection{Maximum principle and energy dissipation.}\label{subsec422}
We take the initial condition $u_{0}(x)=e^{-x^2}$ in $\Omega=(-10,10)$.  The other parameters are chosen as $\epsilon=0.01$, $T=100$, $\tau=10^{-2}$, and degree of freedom $N=2^{12}$. In Figure\,\,\ref{utmaxsdifferent}, we present the evolution of maximum value and energy at various times with $s\leq 0.5$ and $s>0.5$, which shows that the maximum principle is preserved and the energy dissipation law is also justified numerically. We observe from Figure\,\,\ref{utmaxsdifferent} that both the maximum value and the corresponding energy of the steady state are increased as $s$ increases.


\begin{figure}[!ht]
\centering
\subfigure[maximum value with $s\leq 0.5$]{
\includegraphics[width=0.45\textwidth]{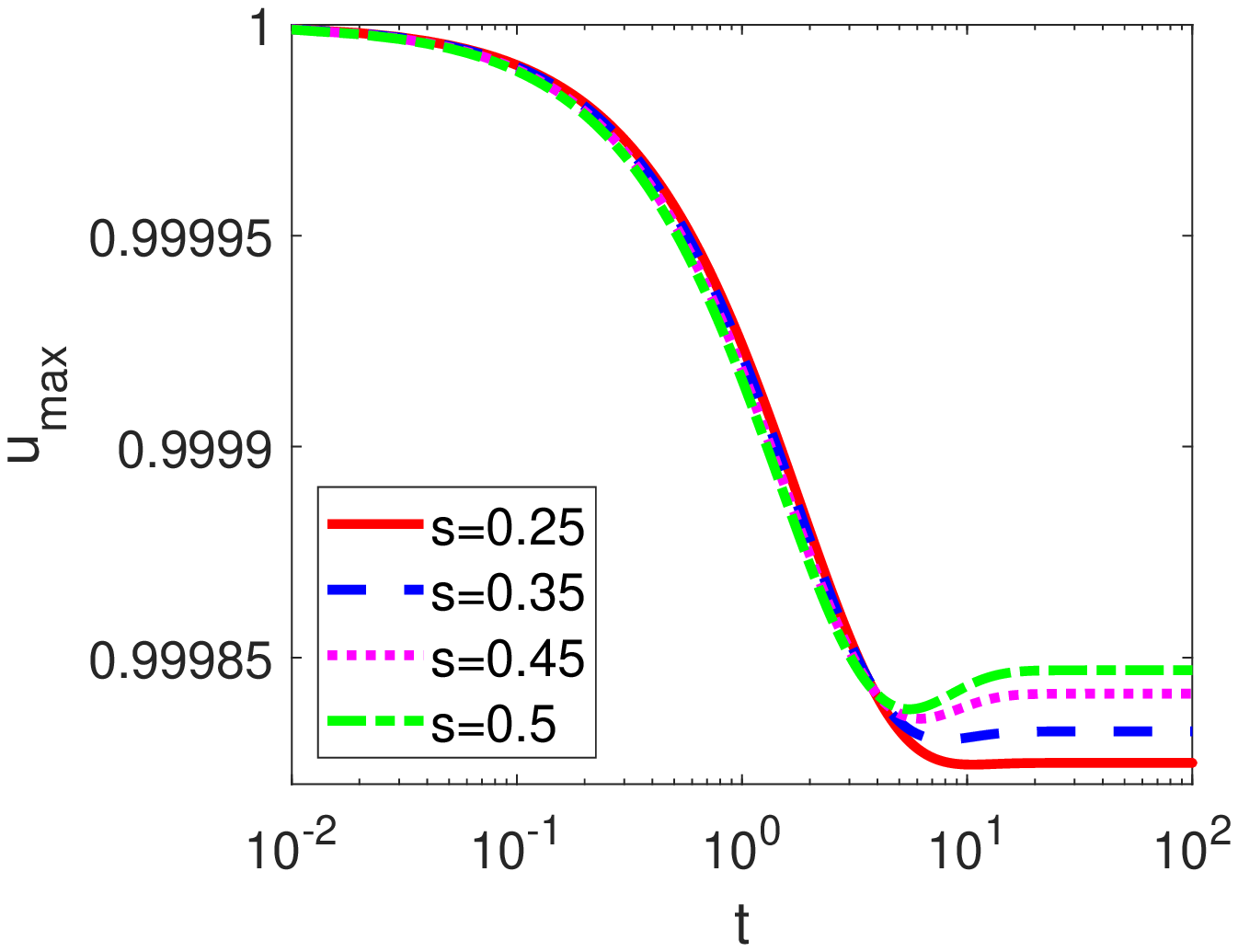}}
\subfigure[maximum value with $s>0.5$]{
\includegraphics[width=0.45\textwidth]{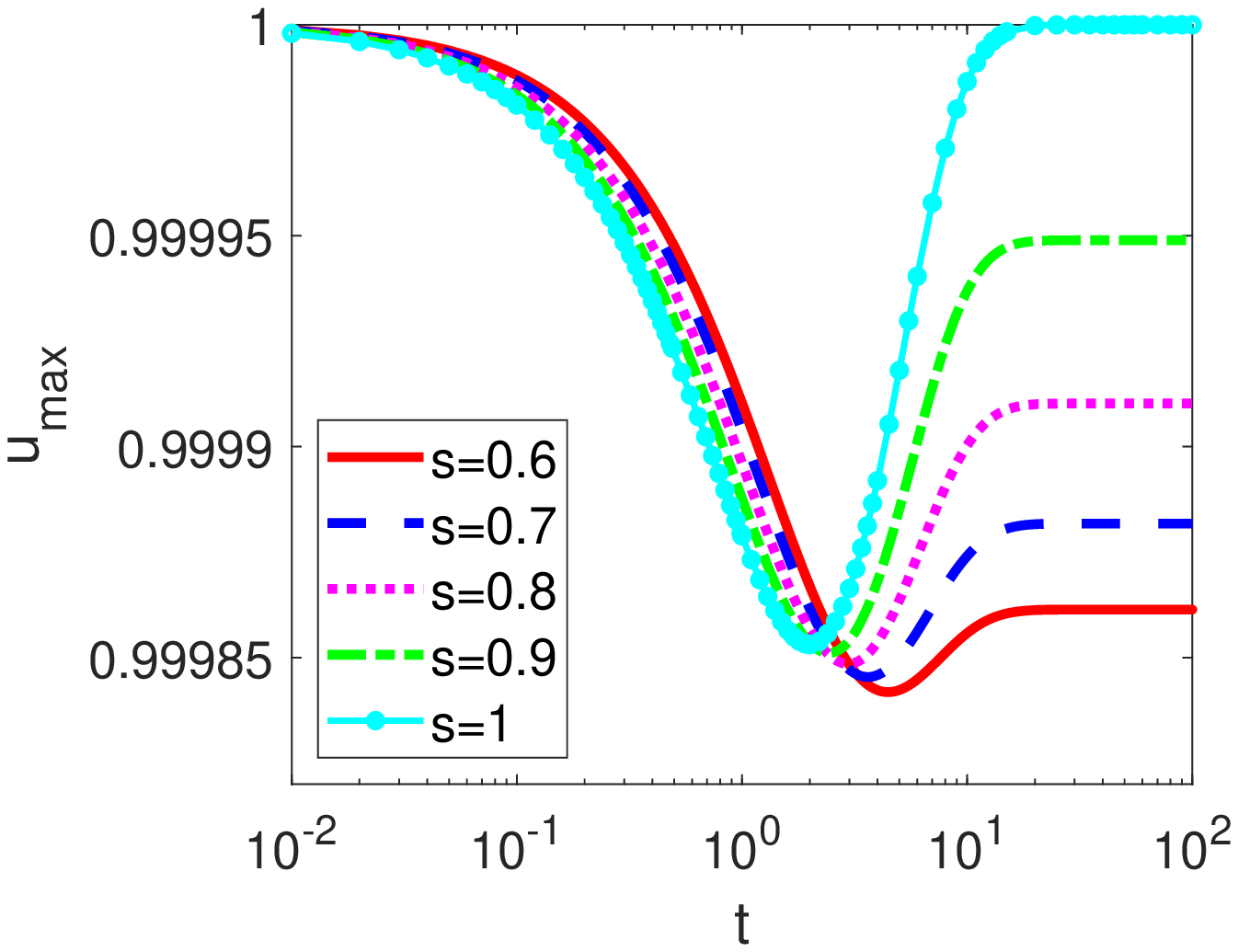}}
\subfigure[energy with $s\leq 0.5$]{
\includegraphics[width=0.45\textwidth]{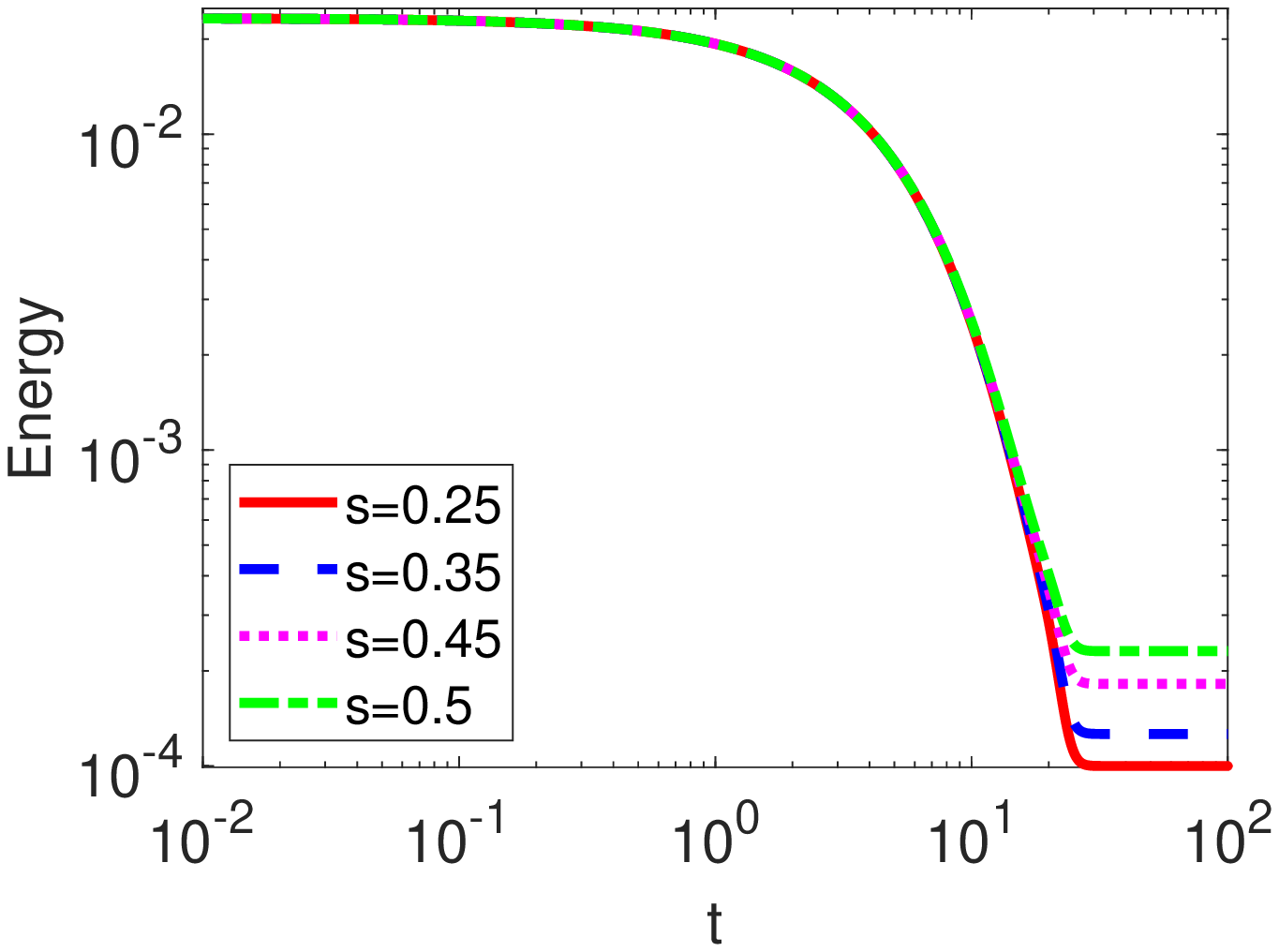}}
\subfigure[energy with $s>0.5$]{
\includegraphics[width=0.45\textwidth]{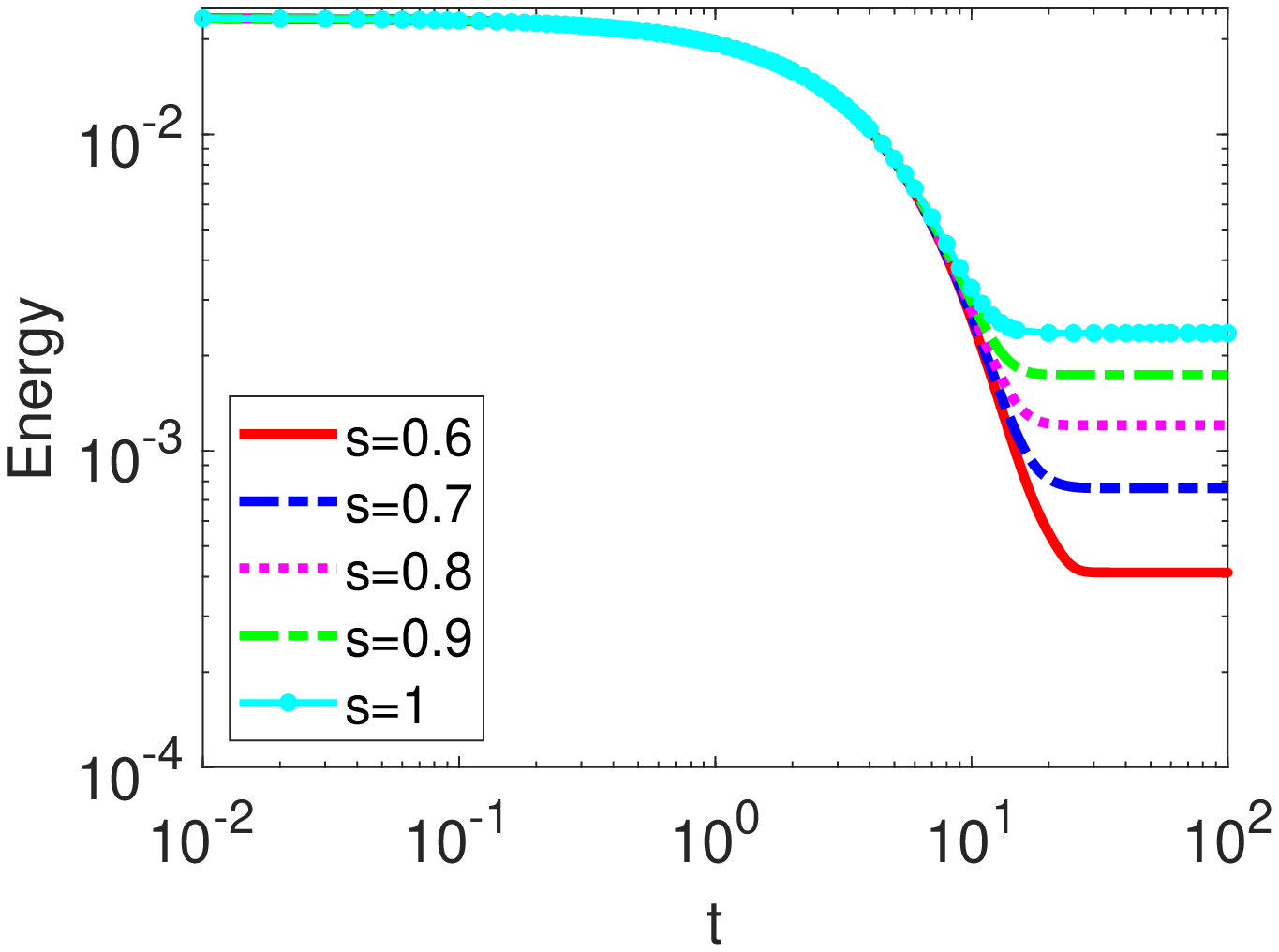}}
\caption{Evolution of maximum value and energy with the initial condition $u_{0}(x)=e^{-x^2}$.}
\label{utmaxsdifferent}
\end{figure}

\begin{figure}[!ht]
\centering
\subfigure[asymptotic behavior with $s\leq 0.5$]{
\includegraphics[width=0.45\textwidth]{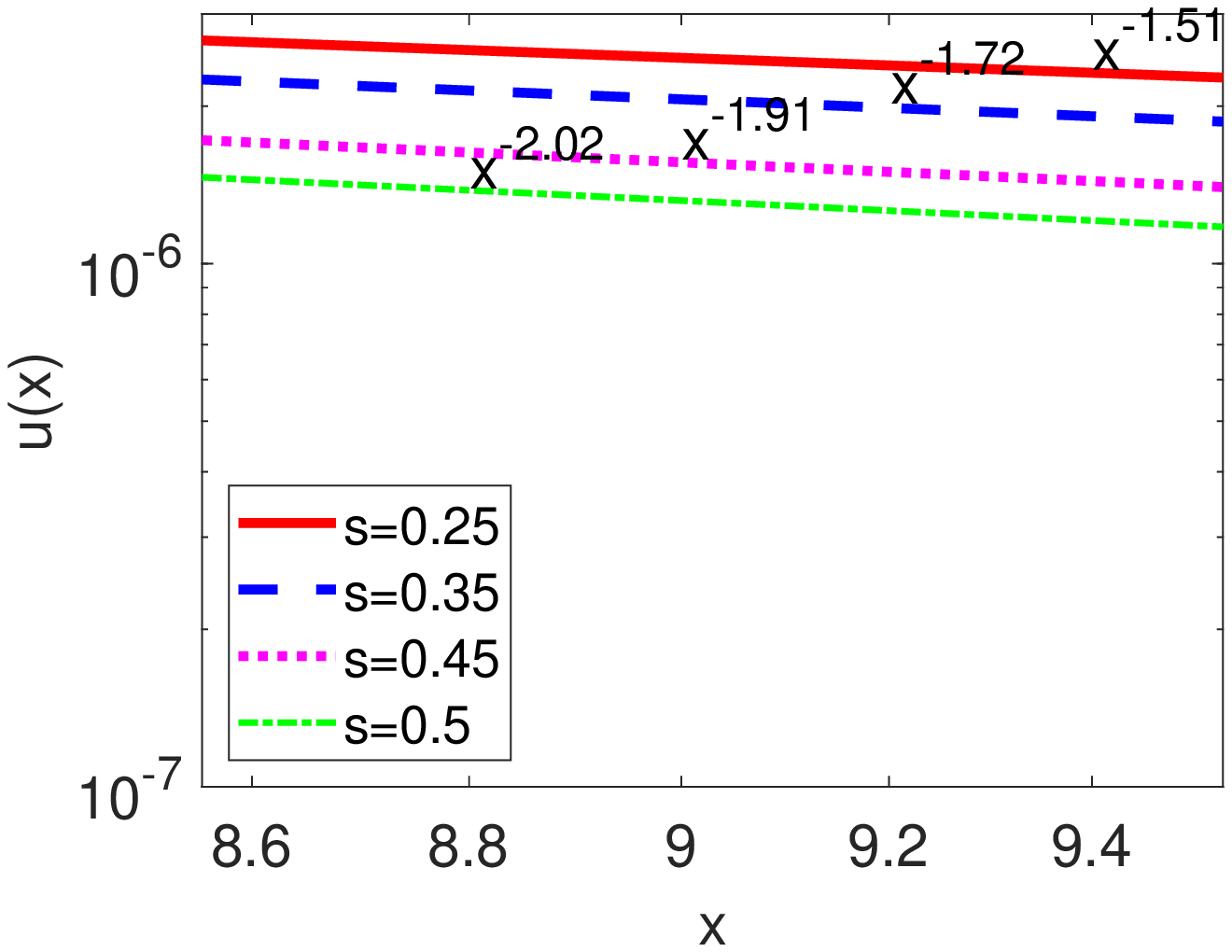}}
\subfigure[asymptotic behavior with $s>0.5$]{
\includegraphics[width=0.45\textwidth]{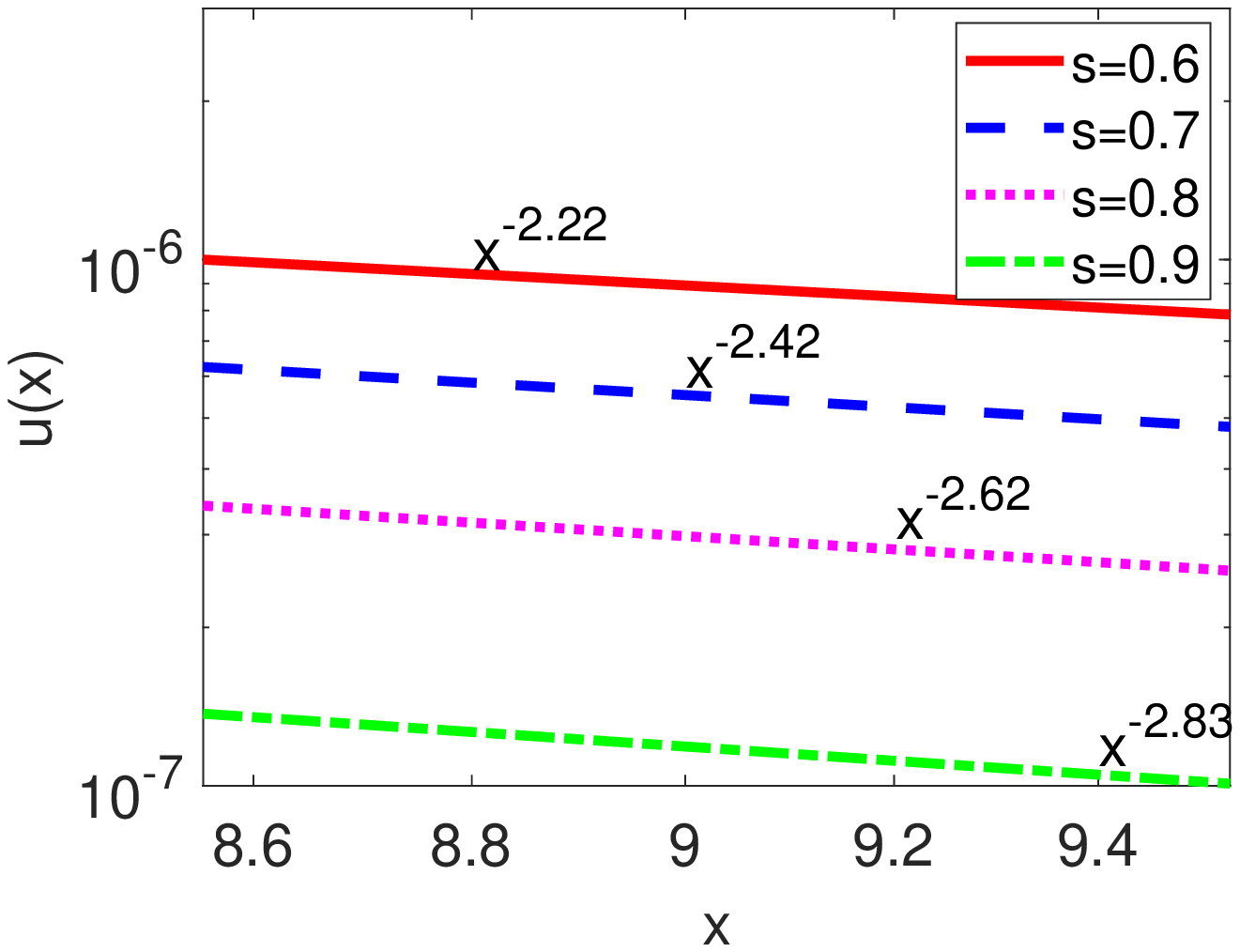}}
\caption{ (a) asymptotic behavior with $s\leq 0.5$; (b) asymptotic behavior with $s>0.5$.}
\label{utasymsdifferent}
\end{figure}
\subsubsection{Asymptotic behavior and interfacial layer}
In order to further study the asymptotic and interfacial behavior of the solution, we still use the same initial condition and parameters as in section \ref{subsec422}.
The numerical results for the asymptotic behavior when $|x|$ is relatively large at time $T=100$ are presented in Figure\,\,\ref{utasymsdifferent}. We observe from Figure\,\,\ref{utasymsdifferent} that the decay property of $u(x,T)$ is slightly different from the integer case (i.e., $s=1$), where $u(x,T)$ decay exponentially as $|x|\to \infty$ when $s=1$. Instead, the solution $u(x,T)$ with $s\in(0,1)$ decay algebraically and behaves like $u(x,T)\sim |x|^{-(2s+1)}$ when $|x|$ is relatively large. We also plot the interfacial layer in Figure\,\,\ref{utainterfacebeh}. We find that the transition of $u(x)$ from $1$ to $0$ will become smoother as $s$ increases, and will become steeper as $s$ decreases.  It is worthwhile to point out that the fractional PDEs with small $s$ should be a powerful tool to simulate the problem with a very steep interface.

\begin{figure}[!ht]
\centering
\subfigure[interfacial layer with $s\leq 0.5$]{
\includegraphics[width=0.45\textwidth]{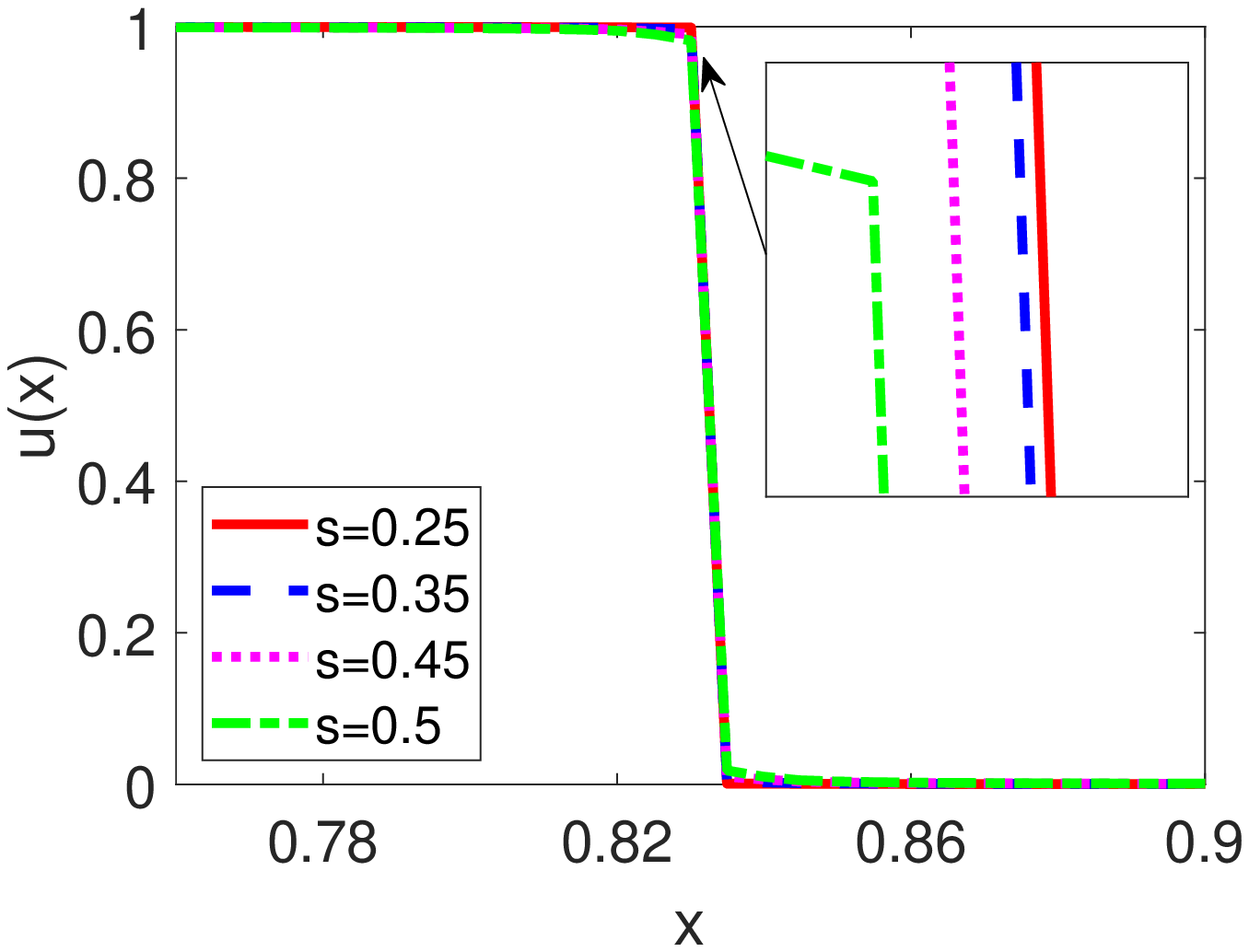}}
\subfigure[interfacial layer with $s>0.5$]{
\includegraphics[width=0.45\textwidth]{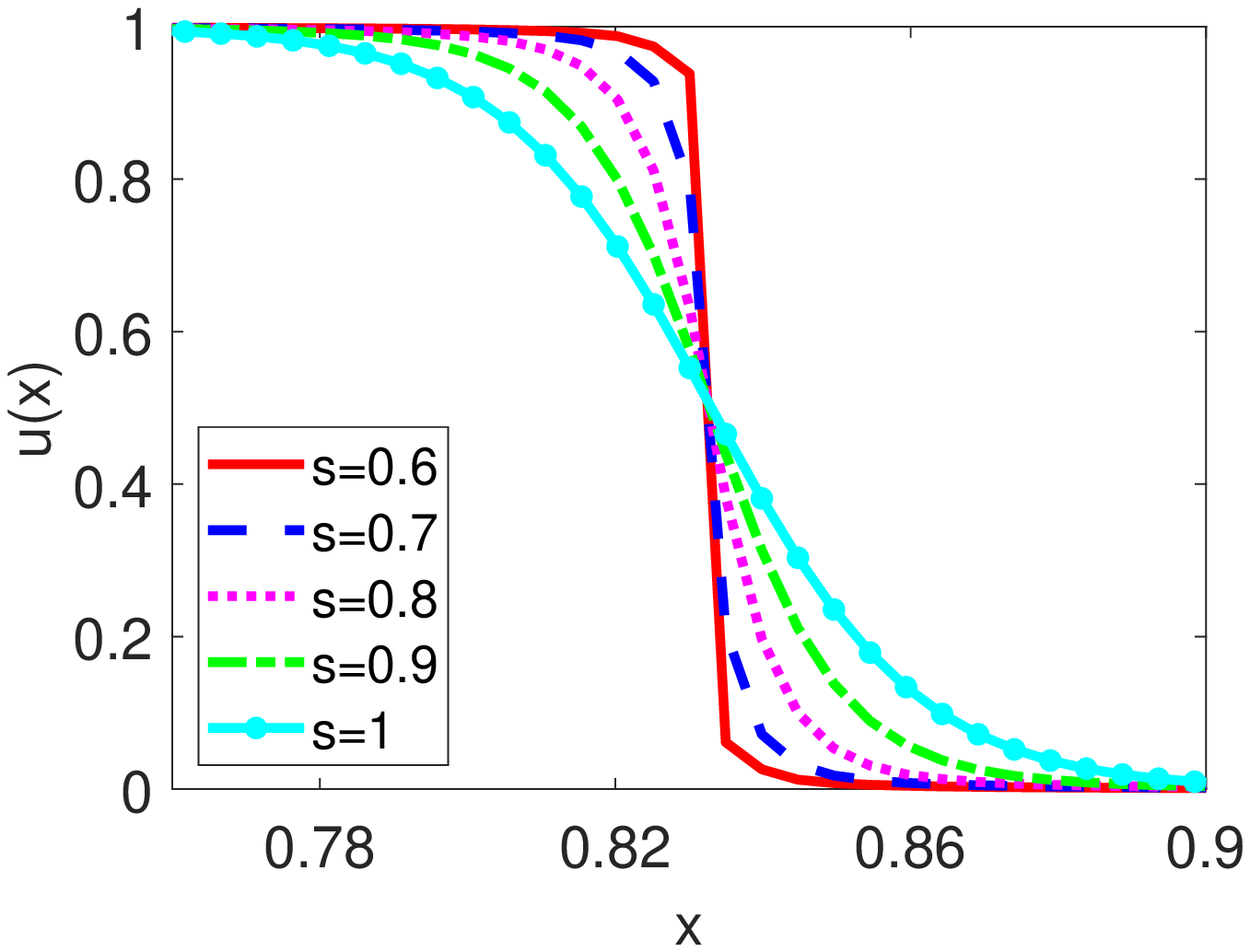}}
\caption{ (a) interfacial layer for $x$ is small with $s\leq 0.5$; (b) interfacial layer for $x$ is small with $s>0.5$.}
\label{utainterfacebeh}
\end{figure}

\begin{figure}[!ht]
\centering
\subfigure[$s=0.8$]{
\includegraphics[width=0.45\textwidth]{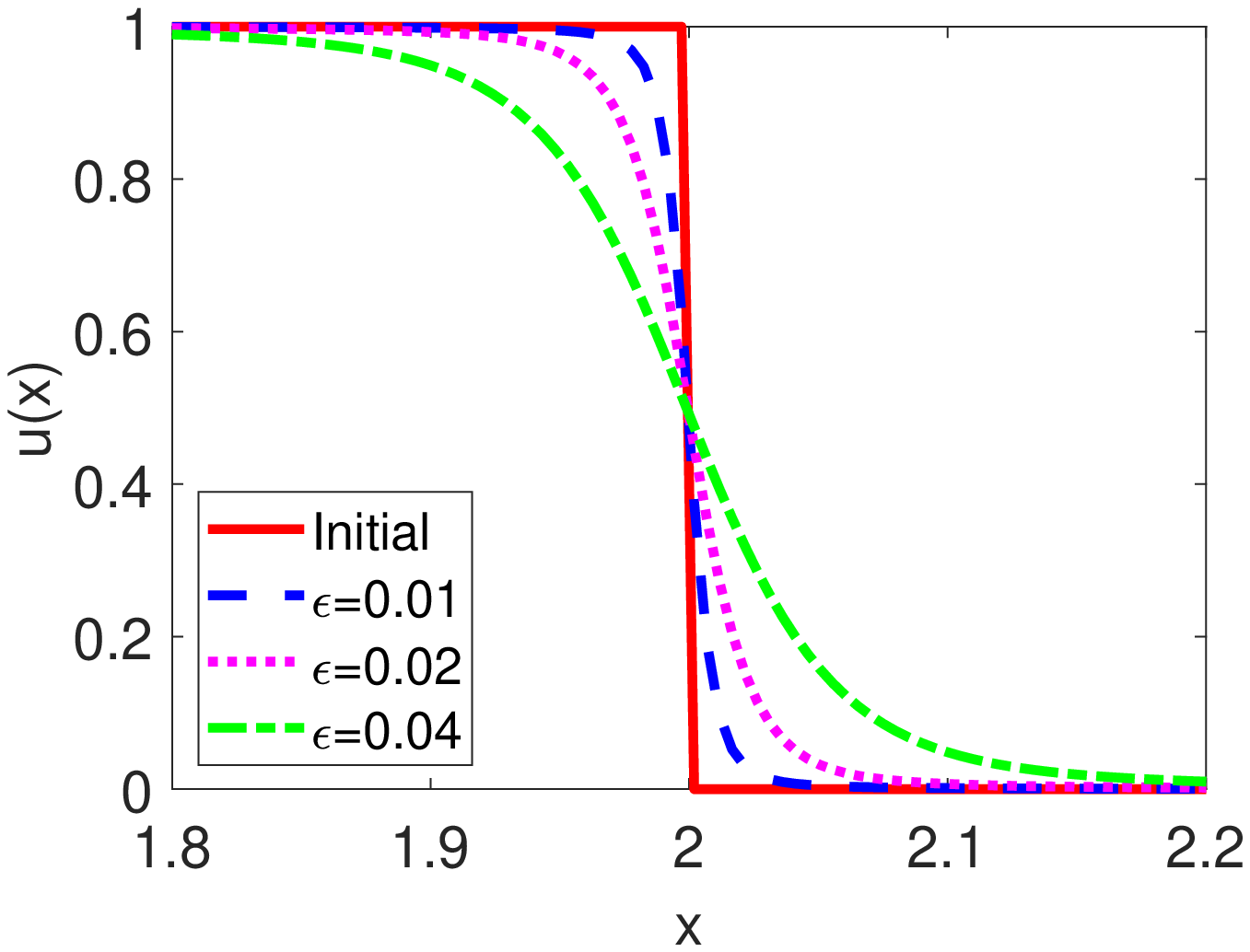}}
\subfigure[different $s$]{
\includegraphics[width=0.45\textwidth]{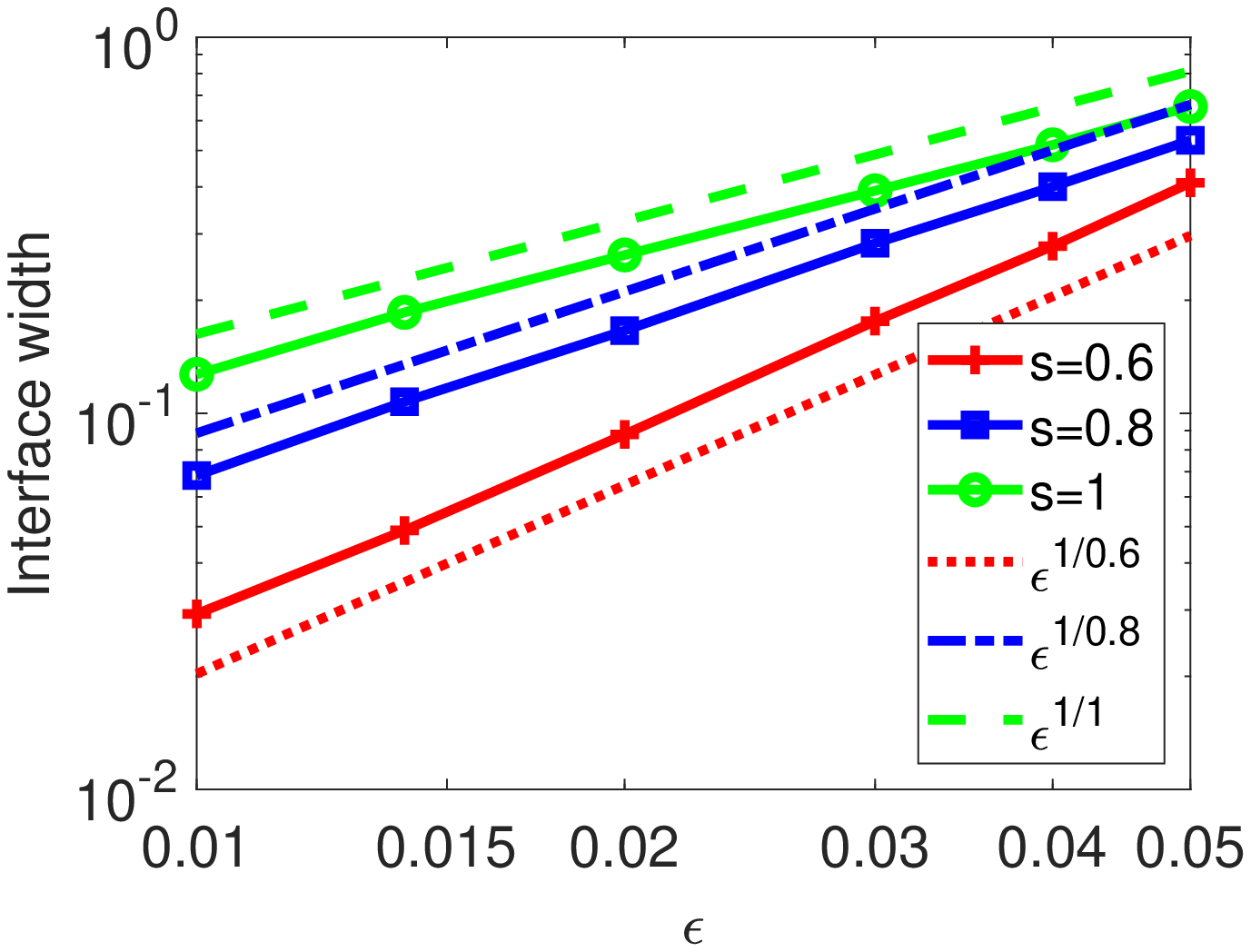}}
\caption{(a) interfacial layer against various $\epsilon$ with $s=0.8$ zoomed in $[1.8,2.2]$. (b) interfacial width agsinst various $\epsilon$ and $s$.}
\label{interface}
\end{figure}
\subsubsection{Interfacial width}
It is well-known that the parameter $\epsilon$ represents the interfacial width of the classical Allen-Cahn equation, i.e., $s=1$,  while for fractional-in-space Allen-Cahn equation the interfacial width decreases as $s$ decreases (cf. \cite{burrage2012efficient,song2016fractional,sheng2020efficient}). However, those existing work enjoys fractional-in-space Allen-Cahn equation with Riemann-Liouville fractional derivatives or spectral fractional Laplacian operator instead of the integral fractional Laplacian. As we mentioned before, different definitions of fractional Laplacian operators are very different from each other. In this example, we take the initial condition to be $u_0(x)=1$ if $x\in(-2,2)$ and $u_0(x)=0$ on $\Omega\backslash (-2,2)$, where $\Omega=(-10,10)$.
The other parameters are $T=100$, $\tau=10^{-2}$, and degree of freedom $N=2^{12}$. We plot interfacial layer with various $\epsilon$ in Figure\,\,\ref{interface} (a), for which we take $s=0.8$. We observe that the  interfacial layer is smoother for larger $\epsilon$. We then present in Figure\,\,\ref{interface} (b) the interfacial width with various $\epsilon$  and $s$. Here, the interfacial width is calculated by the distance of two points where $u(x,T)$ first exceeds $0.01$ and $0.99$. We observe from Figure\,\,\ref{interface} (b) that interfacial width increases as $s$ increases, and behaves like $O(\epsilon^{1/s})$.

\section{Concluding remarks}
In this paper, we derive the explicit form of the stiffness matrix associate with the integral fractional Laplacian in the frequency space, which is extendable to multi-dimensional rectangular elements.  Then we give a complete answer  to the question on when the stiffness matrix can be strictly diagonally dominant.  As an application, we consider the fractional-in-space Allen-Cahn equation and show that it satisfies the maximum principle and energy dissipation law at the continuous level. Then, we proposed two full-discrete  schemes using the  semi-implicit and Crank-Nicolson scheme in time and modified FEM in space,  which can preserve these two properties. Our numerical experiments demonstrate that our algorithms are efficient and accurate. We also observe some interesting phenomena related to  the transition of the model when the fractional order $s\in (s_0,1)$ varies to $s=1$. For example, the
width of the interfacial layer behaves like   $O(\epsilon^{1/s}),$ and the decay of the solution in space obeys certain power law as reported earlier.

\begin{appendix}
 \renewcommand{\theequation}{A.\arabic{equation}}
\section{Proof of Theorem \ref{Dunfordfemp1}} \label{AppendixA}

Using \eqref{freqform} and  Lemma \ref{Dunfordfemp1},  we obtain from direct calculation and a change of variable that
\begin{equation}\begin{split}\label{femp1}
S_{kj}&=\frac{2}{\pi h^2}\int_{\mathbb R} |\xi|^{2s} \cos((k-j)h\xi)\Big(\frac{1-\cos(h\xi)}{\xi^2}\Big)^2\,{\rm d}\xi\\&
=\frac{4}{\pi h^2}\int_0^\infty   \xi^{2s} \cos((k-j)h\xi)\Big(\frac{1-\cos(h\xi)}{\xi^2}\Big)^2\,{\rm d}\xi\\&
=\frac{4}{\pi h^{2s-1}}\int_0^\infty y^{2s-4} \cos(|k-j| y)(1-\cos y)^2\,{\rm d}y,
\end{split}\end{equation}
which implies the entry $S_{kj}$ only depends on $p=|k-j|,$ so the matrix $\bs S$ is a symmetric  Toeplitz matrix.  We intend to explicitly evaluate the integral \eqref{femp1}.  From the fundamental trigonometric identities, we find readily that
\begin{equation}
\begin{split}\label{fyb}
f(y;p)&:=\cos(p y)(1-\cos y)^2=\cos(p y)\Big(\frac32-2\cos y+\frac12\cos(2y)\Big)
 \\& =\frac1 4 \sum_{i=-2}^2  c_i \cos(|p+i|y).
\end{split}
\end{equation}
We continue the calculation by using integration by parts.   The number of times that we can integrate by parts  depends on the range of $s,$ so we proceed with three cases with  $s\in (1,\frac 32), (\frac 1 2, 1)$ and $(0,\frac 1 2), $ separately.  Then we derive the formulas for  $s=\frac 1 2, 1$  by taking limits.

\medskip
\underline{\bf  Case (i)  $s\in(1,\frac32)$:}\,  Recall the integral identity (cf. \cite[P. 440]{Gradshteyn2015Book}):
\begin{equation}
\label{sinint0}
\int^\infty_0x^{\mu-1}\sin(ax)\, {\rm d}x=\frac{\Gamma(\mu)}{a^\mu}\sin\Big(\frac{\mu\pi}{2}\Big),\quad a>0, \;\; \mu\in(0,1).
\end{equation}
We derive from \eqref{femp1} and integration by parts immediately  that
\begin{equation}\label{smn3L}
\int_0^\infty y^{2s-4}f(y;p){\rm d}y= 
-\frac{1}{2s-3}\int_0^\infty y^{2s-3}f^\prime(y;p){\rm d}y.
\end{equation}
By \eqref{fyb}, we obtain from \eqref{sinint0} with $\mu=2s-2\in (0,1)$  that
\begin{equation}\label{smn4L}
\begin{split}
\int_0^\infty y^{2s-3}& f^{\prime}(y;p){\rm d}y =-\frac14 \int_0^\infty  y^{2s-3}\big\{|p-2|\sin(|p-2|y)-4|p-1|\sin(|p-1|y)
\\&\quad+6p\sin(p y)-4(p+1)\sin((p+1)y)+(p+2)\sin((p+2)y)\big\} {\rm d}y
\\& =-\frac {\Gamma(2s-2)} 4\sin\big((s-1)\pi\big)\big\{|p-2|^{3-2s}-4|p-1|^{3-2s}
+6p^{3-2s}\\
&\quad -4(p+1)^{3-2s}+(p+2)^{3-2s}\big\}.
\end{split}
\end{equation}
Thus,  we infer from \eqref{femp1}-\eqref{fyb} and  \eqref{smn3L}-\eqref{smn4L} that
\begin{equation*}
\begin{split}
S_{kj} & 
=\frac{1}{\pi h^{2s-1}}
\frac{1}{2s-3}\int_0^\infty y^{2s-3}f^{\prime}(y;p)\,{\rm d}y=-\frac{ \sin \big(s\pi\big) \;  \Gamma(2s-3)} {\pi h^{2s-1}} t_p.
\end{split}
\end{equation*}
Then  we  have the entries of $\bs S$ with $s\in (1,\frac 32)$ by using  the  reflection property (cf. \cite[P. 138]{Nist2010}):
\begin{equation}\label{gammareflection}
\Gamma(z)\Gamma(1-z)=\frac \pi {\sin \pi z},\quad z\not=0,-1,-2,\cdots.
\end{equation}

\smallskip

\underline{\bf  Case (ii)  $s\in(\frac12,1)$:}\,  In this case, we can integrate \eqref{smn3L} by parts
 one more time, and then use the identity (cf. \cite[P. 441]{Gradshteyn2015Book}):
\begin{equation}
\label{cosint}
\int^\infty_0x^{\mu-1}\cos(ax)\,{\rm d}x=\frac{\Gamma(\mu)}{a^\mu}\cos\Big(\frac{\mu\pi}{2}\Big),\quad a>0,\;\; \; \mu\in(0,1).
\end{equation}
More precisely, by \eqref{smn3L} and \eqref{cosint} with $\mu=2s-1\in (0,1),$ 
\begin{equation}
\begin{split}\label{smn3}
&\int_0^\infty y^{2s-4} f(y;p){\rm d}y=\frac{1}{(2s-3)(2s-2)}\int_0^\infty y^{2s-2}f^{\prime\prime}(y;p){\rm d}y\\
&=-\frac1{4(2s-3)(2s-2)} \int_0^\infty  y^{2s-2}\big\{|p-2|^2\cos(|p-2|y)-4|p-1|^2\cos(|p-1|y)
\\&\quad+6p^2\cos(p y)-4(p+1)^2\cos((p+1)y)+(p+2)^2\cos((p+2)y)\big\} {\rm d}y
\\&=-\frac {\Gamma(2s-3)} 4\cos\Big(\frac{(2s-1)\pi}{2}\Big)t_p.
\end{split}
\end{equation}
Hence,  by \eqref{femp1} and  \eqref{gammareflection},
\begin{equation}\label{skj01}
\begin{split}
&S_{kj} =\frac{4}{\pi h^{2s-1}}\int_0^\infty y^{2s-4} f(y;p) \,{\rm d}y=-\frac{ \sin (s\pi) \;  \Gamma(2s-3)} {\pi h^{2s-1}} t_p=\frac{ \sec(s\pi)} {2h^{1-2s} \Gamma(4-2s) } t_{p},
\end{split}
\end{equation}
which yields the desired formula  with $s\in (\frac 12,1).$

%

\smallskip

\underline{\bf  Case (iii)  $s\in(0, \frac12)$:}\, Similarly, we integrate the first equality  of  \eqref{smn3}  by parts one more time  and obtain
\begin{equation}
\begin{split}\label{smn1}
\int_0^\infty y^{2s-4}f(y;p){\rm d}y &=
-\frac{1}{(2s-3)(2s-2)(2s-1)}\int_0^\infty y^{2s-1}f^{\prime\prime\prime}(y;p)\,{\rm d}y,
\end{split}
\end{equation}
which involves  sines, so we use  \eqref{sinint0} with  $\mu=2s \in (0,1)$ to evaluate the integrals.
Then following the same lines as the previous case, we can derive the entries with $s\in (0, \frac 12)$ in a similar fashion.


\smallskip

\underline{\bf  Case (iv)  $s=\frac 1 2$:}\,
Observe from  \eqref{freqform} that $S_{kj}$  continuously depends on the parameter $s.$
 We can compute the entries by taking the limit
\begin{equation}\label{newlimits}
S_{kj}= \frac 1 4 \lim_{s\to \frac 1 2} \frac{t_p}{\cos(s\pi)}, 
\end{equation}
and  resort to the basic limit
\begin{equation}\label{LLWang1}
\ln z = \lim_{\epsilon\rightarrow 0}\frac{z^{\epsilon} - 1}{\epsilon},\ \ z>0,
\end{equation}
which is a direct consequence of  the L'Hospital's rule. Note that
\begin{equation}\label{psum0}
\sum_{i=-2}^2c_i (p+i)^l=0,\quad l=0,1,2,
\end{equation}
so we can rewrite the entries in \eqref{gpa00} for $p = |k-j|\geq 3$ as 
\begin{equation}\label{newlimits-1}
\begin{split}
S_{kj} & = \frac 1 {4} \lim_{s\to \frac 1 2} \frac{1-2s}{\cos(s\pi)}
 \sum_{i=-2}^2 c_i \lim_{s\to \frac 1 2} \frac{|p+i|^{3-2s} -(p+i)^2}{1-2s}\\
 &=\frac 1 {2\pi} \sum_{i=-2}^2 c_i (p+i)^2  \lim_{s\to \frac 1 2} \frac{|p+i|^{1-2s} -1}{1-2s}
= \frac 1 {2\pi} \sum_{i=-2}^2 c_i (p+i)^2 \ln |p+i|.
\end{split}
\end{equation}
Note that for  $p=0,1,2$ and $p+i=0$, the formula still holds with the understanding of $(p+i)^2\ln|p+i|=0$. This leads to \eqref{StiffMatrix01/2}.

\smallskip

\underline{\bf (v). Case $s=1$:} In this case, we directly take the limit upon \eqref{gpa00}, and find readily that
 \begin{equation}\label{direcal}
 S_{kj} 
 =-\frac {t_p} {2h}\Big|_{s=1}
 =\frac{1}{h} \begin{cases}{2,} & {j=k}, \\ {-1,} & {j=k \pm 1}, \\ {0,} & {\text { otherwise. }}
 \end{cases}
 \end{equation}
 This yields the stiffness matrix of the usual 1D Laplacian as expected.

\end{appendix}

\bibliographystyle{siam}

\bibliography{reffractional}

\begin{thebibliography}{10}

\bibitem{acosta2017short}
G.~Acosta, F.M. Bersetche, and J.P. Borthagaray.
\newblock A short {FE} implementation for a 2d homogeneous {D}irichlet problem
  of a fractional {L}aplacian.
\newblock {\em Comput. Math. Appl.}, 74(4):784--816, 2017.

\bibitem{Ainsworth2018}
M.~Ainsworth and C.~Glusa.
\newblock Towards an efficient finite element method for the integral
  fractional {L}aplacian on polygonal domains.
\newblock In {\em Contemporary computational mathematics---a celebration of the
  80th birthday of {I}an {S}loan. {V}ol. 1, 2}, pages 17--57. Springer, Cham,
  2018.

\bibitem{Alfa2002}
A.S. Alfa, J.G. Xue, and Q.~Ye.
\newblock Accurate computation of the smallest eigenvalue of a diagonally
  dominant {$M$}-matrix.
\newblock {\em Math. Comp.}, 71(237):217--236, 2002.

\bibitem{babuska2001finite}
I.~Babu{\v{s}}ka, T.~Strouboulis, et~al.
\newblock {\em The finite element method and its reliability}.
\newblock Oxford university press, 2001.

\bibitem{bueno2014fourier}
A.~Bueno-Orovio, D.~Kay, and K.~Burrage.
\newblock Fourier spectral methods for fractional-in-space reaction-diffusion
  equations.
\newblock {\em BIT Numer. Math.}, 54(4):937--954, 2014.

\bibitem{burrage2012efficient}
K.~Burrage, N.~Hale, and D.~Kay.
\newblock An efficient implicit {FEM} scheme for fractional-in-space
  reaction-diffusion equations.
\newblock {\em SIAM J. Sci. Comput.}, 34(4):A2145--A2172, 2012.

\bibitem{du2020phase}
Q.~Du and X.B. Feng.
\newblock The phase field method for geometric moving interfaces and their
  numerical approximations.
\newblock In {\em Handbook of Numerical Analysis}, volume~21, pages 425--508.
  Elsevier, 2020.

\bibitem{du2019time}
Q.~Du, J.~Yang, and Z.~Zhou.
\newblock Time-fractional {A}llen-{C}ahn equations: Analysis and numerical
  methods.
\newblock {\em arXiv:1906.06584}, 2019.

\bibitem{duo2018novel}
S.W. Duo, H.W. van Wyk, and Y.Z. Zhang.
\newblock A novel and accurate finite difference method for the fractional
  {L}aplacian and the fractional poisson problem.
\newblock {\em J. Comput. Phys.}, 355:233--252, 2018.

\bibitem{duo2019fractional}
S.W. Duo and H.~Wang.
\newblock A fractional phase-field model using an infinitesimal generator of
  $\alpha$ stable {L}{\'e}vy process.
\newblock {\em J. Comput. Phys.}, 384:253--269, 2019.

\bibitem{duo2019accurate}
S.W. Duo and Y.Z. Zhang.
\newblock Accurate numerical methods for two and three dimensional integral
  fractional {L}aplacian with applications.
\newblock {\em Comput. Methods Appl. Mech. Eng.}, 355:639--662, 2019.

\bibitem{MR698779}
A.~Erd\'{e}lyi, W.~Magnus, F.~Oberhettinger, and F.G. Tricomi.
\newblock {\em Higher transcendental functions. {V}ol. {I}}.
\newblock Robert E. Krieger Publishing Co., Inc., Melbourne, Fla., 1981.
\newblock Based on notes left by Harry Bateman, With a preface by Mina Rees,
  With a foreword by E. C. Watson, Reprint of the 1953 original.

\bibitem{farid2011notes}
F.O. Farid.
\newblock Notes on matrices with diagonally dominant properties.
\newblock {\em Linear Algebra Appl.}, 435(11):2793--2812, 2011.

\bibitem{feng2003numerical}
X.B. Feng and A.~Prohl.
\newblock Numerical analysis of the {A}llen-{C}ahn equation and approximation
  for mean curvature flows.
\newblock {\em Numer. Math.}, 94(1):33--65, 2003.

\bibitem{george2004gaussian}
A.~George and K.D. Ikramov.
\newblock Gaussian elimination is stable for the inverse of a diagonally
  dominant matrix.
\newblock {\em Math. Comp.}, 73(246):653--657, 2004.

\bibitem{Golub1996}
G.H. Golub and C.F. Van~Loan.
\newblock {\em Matrix Computations}.
\newblock Johns Hopkins Studies in the Mathematical Sciences. Johns Hopkins
  University Press, Baltimore, MD, third edition, 1996.

\bibitem{Gradshteyn2015Book}
I.S. Gradshteyn and I.M. Ryzhik.
\newblock {\em Table of Integrals, Series, and Products}.
\newblock Elsevier/Academic Press, Amsterdam, eighth edition, 2015.
\newblock Translated from the Russian, Translation edited and with a preface by
  Daniel Zwillinger and Victor Moll, Revised from the seventh edition
  [MR2360010].

\bibitem{Roger1985}
R.A. Horn and C.R. Johnson.
\newblock {\em Matrix Analysis}.
\newblock Cambridge University Press, Cambridge, 1985.

\bibitem{hou2017numerical}
T.L. Hou, T.~Tang, and J.~Yang.
\newblock Numerical analysis of fully discretized {C}rank--{N}icolson scheme
  for fractional-in-space {A}llen--{C}ahn equations.
\newblock {\em J. Sci. Comput.}, 72(3):1214--1231, 2017.

\bibitem{li2015numerical}
Y.K. Li.
\newblock Numerical methods for deterministic and stochastic phase field models
  of phase transition and related geometric flows ({Ph}.{D}. thesis).
\newblock {\em University of Tennessee}, 2015.

\bibitem{li2017space}
Z.~Li, H.~Wang, and D.P. Yang.
\newblock A space--time fractional phase-field model with tunable sharpness and
  decay behavior and its efficient numerical simulation.
\newblock {\em J. Comput. Phys.}, 347:20--38, 2017.

\bibitem{liao2019second}
H.L. Liao, T.~Tang, and T.~Zhou.
\newblock A second-order and nonuniform time-stepping maximum-principle
  preserving scheme for time-fractional {A}llen-{C}ahn equations.
\newblock {\em J. Comput. Phys.}, 414:109473, 2020.

\bibitem{liu2018time}
H.~Liu, A.J. Cheng, H.~Wang, and J.~Zhao.
\newblock Time-fractional {A}llen--{C}ahn and {C}ahn--{H}illiard phase-field
  models and their numerical investigation.
\newblock {\em Comput. Math. Appl.}, 76(8):1876--1892, 2018.

\bibitem{Nist2010}
F.W.J. Olver, D.W. Lozier, R.F. Boisvert, and C.W. Clark, editors.
\newblock {\em N{IST} Handbook of Mathematical Functions}.
\newblock U.S. Department of Commerce, National Institute of Standards and
  Technology, Washington, DC; Cambridge University Press, Cambridge, 2010.
\newblock With 1 CD-ROM (Windows, Macintosh and UNIX).

\bibitem{Quarteroni94}
A.~Quarteroni and A.~Valli.
\newblock {\em {Numerical Approximation of Partial Differential Equations}},
  volume~23 of {\em Springer Series in Computational Mathematics}.
\newblock Springer-Verlag, Berlin, 1994.

\bibitem{ShenTangWang2011}
J.~Shen, T.~Tang, and L.L. Wang.
\newblock {\em {Spectral Methods: Algorithms, Analysis and Applications}},
  volume~41 of {\em Series in Computational Mathematics}.
\newblock Springer-Verlag, Berlin, Heidelberg, 2011.

\bibitem{shen2019new}
J.~Shen, J.~Xu, and J.~Yang.
\newblock A new class of efficient and robust energy stable schemes for
  gradient flows.
\newblock {\em SIAM Rev.}, 61(3):474--506, 2019.

\bibitem{sheng2020efficient}
C.T. Sheng, D.~Cao, and J.~Shen.
\newblock Efficient spectral methods for {PDEs} with spectral fractional
  {L}aplacian.
\newblock {\em Submitted}, 2020.

\bibitem{sheng2019fast}
C.T. Sheng, J.~Shen, T.~Tang, L.L. Wang, and H.F. Yuan.
\newblock Fast {F}ourier-like mapped {C}hebyshev spectral-{G}alerkin methods
  for {PDE}s with integral fractional {L}aplacian in unbounded domains.
\newblock {\em Accepted by SIAM J. Numer. Anal.}, 2020.

\bibitem{song2016fractional}
F.Y. Song, C.J. Xu, and G.~E. Karniadakis.
\newblock A fractional phase-field model for two-phase flows with tunable
  sharpness: {A}lgorithms and simulations.
\newblock {\em Comput. Methods Appl. Mech. Engrg.}, 305:376--404, 2016.

\bibitem{tang2020revisit}
T.~Tang.
\newblock Revisit of semi-implicit schemes for phase field equation.
\newblock {\em arXiv:2006.06990}, 2020.

\bibitem{tang2016implicit}
T.~Tang and J.~Yang.
\newblock Implicit-explicit scheme for the {A}llen-{C}ahn equation preserves
  the maximum principle.
\newblock {\em J. Comput. Math.}, 34(5):471--481, 2016.

\bibitem{tang2019energy}
T.~Tang, H.J. Yu, and T.~Zhou.
\newblock On energy dissipation theory and numerical stability for
  time-fractional phase-field equations.
\newblock {\em SIAM J. Sci. Comput.}, 41(6):A3757--A3778, 2019.

\bibitem{tian2015class}
W.Y. Tian, H.~Zhou, and W.H. Deng.
\newblock A class of second order difference approximations for solving space
  fractional diffusion equations.
\newblock {\em Math. Comp.}, 84(294):1703--1727, 2015.

\bibitem{Tveito2009}
A.~Tveito and R.~Winter.
\newblock {\em {Introduction to Partial Differential Equations}}, volume~29 of
  {\em Texts in Applied Mathematics}.
\newblock Springer-Verlag, Berlin, 2009.
\newblock A Computational Approach, Paperback reprint of the 2005 edition.

\bibitem{urekew1993importance}
T.J. Urekew and J.J. Rencis.
\newblock The importance of diagonal dominance in the iterative solution of
  equations generated from the boundary element method.
\newblock {\em Int. J. Numer. Meth. Eng.}, 36(20):3509--3527, 1993.

\bibitem{wang2019finite}
F.~Wang, H.~Chen, and H.~Wang.
\newblock Finite element simulation and efficient algorithm for fractional
  {Cahn--Hilliard} equation.
\newblock {\em J. Comput. Appl. Math.}, 356:248--266, 2019.

\bibitem{xu2019stability}
J.C. Xu, Y.K. Li, S.N. Wu, and A.~Bousquet.
\newblock On the stability and accuracy of partially and fully implicit schemes
  for phase field modeling.
\newblock {\em Comput. Methods Appl. Mech. Eng.}, 345:826--853, 2019.

\bibitem{zhao2019power}
J.~Zhao, L.Z. Chen, and H.~Wang.
\newblock On power law scaling dynamics for time-fractional phase field models
  during coarsening.
\newblock {\em Commun. Nonlinear Sci. Numer. Simul.}, 70:257--270, 2019.

\end{thebibliography}

\end{document}